\documentclass{article}
\usepackage[sort,square, comma, numbers]{natbib}
\usepackage{amsmath, amsthm, amsfonts, amscd}
\usepackage{geometry}
\usepackage{CJKutf8} 
\usepackage{multirow}
\usepackage{complexity}
\usepackage{amssymb}
\usepackage[hypcap]{caption}
\usepackage[nottoc,numbib]{tocbibind}
\usepackage{mathrsfs}
\usepackage{float}
\usepackage{graphicx}
\usepackage{microtype}
\usepackage{verbatim}
\usepackage{hyperref}
\usepackage{mdwlist}
\usepackage{mathabx}       
\usepackage{mathtools}
\usepackage{titlesec}
\usepackage{tikz,pgf, tikz-cd}
\usepackage[linesnumbered,lined,boxed,commentsnumbered]{algorithm2e}
\usepackage{fancyhdr}
\usepackage[thinlines]{easytable} %
\usepackage[shortlabels]{enumitem}
\usepackage{placeins}
\usepackage[title]{appendix}
\usepackage{booktabs} 

\usetikzlibrary{shapes,decorations,arrows,calc,arrows.meta,fit,positioning}
\usetikzlibrary{decorations.markings}
\usetikzlibrary{arrows, automata}
\usetikzlibrary{decorations.pathreplacing}

\newtheorem{theorem}{Theorem}[section]
\newtheorem{corollary}[theorem]{Corollary}
\newtheorem{proposition}[theorem]{Proposition}
\newtheorem{lemma}[theorem]{Lemma}

\newtheorem*{theorem*}{Theorem}

\newtheorem*{prop*}{Proposition}
\newtheorem*{cor*}{Corollary}

\theoremstyle{definition}
\newtheorem{definition}[theorem]{Definition}

\newtheorem{example}[theorem]{Example}

\newtheorem{remark}[theorem]{Remark}

\def\mc{\mathcal}
\def\mbb{\mathbb}

\def\End{{\it End}_{\Lambda}}
\def\Sym{{\it Sym}}

\def\Ann{{\it Ann}}
\def\L{\Lambda}
\def\l{\lambda}
\def\D{\Delta}

\def\La{\mc{L}_{\it{mod}}}
\def\Lg{\mc{L}_{\it{group}}}
\def\Lal{\mc{L}_{\it{alg}}}
\def\Lr{\mc{L}_{\it{ring}}}

\def\H10{\text{H10}}
\def\mb{\mathbf}

\begin{document}

\title{
%
Studying the Diophantine problem in finitely generated rings and algebras via bilinear maps
}

\author{Albert Garreta\footnote{Basque Center of Applied Mathematics, Bilbao, Spain}        \and
        Alexei Miasnikov\footnote{Stevens Institute of Technology, NJ, USA}  \and 
		Denis Ovchinnikov\footnote{Stevens Institute of Technology, NJ, USA}  
}

\maketitle

\begin{abstract}

We study systems of polynomial equations in several classes of finitely generated rings and algebras.  For each ring $R$ (or algebra) in one of these classes we obtain an interpretation by systems of equations of  a ring of integers $O$ of a finite field extension of either $\mbb{Q}$ or  $\mbb{F}_p(t)$, for some prime $p$ and variable $t$. This implies that the Diophantine problem (decidability of systems of polynomial equations) in $O$ is Karp-reducible to the same problem in $R$. In several cases  we further obtain an interpretation by systems of equations of the ring $\mbb{F}_p[t]$ in $R$, which implies that the Diophantine problem  in $R$ is undecidable in this case.  Otherwise,  the ring  $O$ is a ring of algebraic integers, and then the long-standing conjecture that $\mathbb{Z}$ is always interpretable by systems of equations in a ring of algebraic integers carries over to $R$. If true, it implies that the Diophantine problem in $R$ is also undecidable. 

Some of the classes of finitely generated rings  studied in this paper are the following: all associative, commutative, non-unitary rings (a similar statement for the unitary case was obtained by Eisentraeger); all possibly non-associative, non-commutative non-unitary rings that are finitely generated as an abelian group; and several classes of finitely generated non-commutative rings. Analogous statements are obtained for algebras over finitely generated associative commutative unitary rings.

Another contribution of the paper is the technique by which the aforementioned results are obtained. More precisely, we show that given a bilinear map $f: A\times B \to C$  between finitely generated abelian groups (or modules), under some mild assumptions, there exists a certain ring (or algebra) $R$ with nice properties which is interpretable by systems of equations in the multi-sorted structure $(A,B,C;f)$. This result fits nicely the study of rings (or  algebras)  since the multiplication operation of such structures can be seen as a bilinear map between abelian groups (or modules). This result is potentially  applicable in many other settings, such as in the area of group theory, see for example \cite{GMO_solvable}.


\end{abstract}

\setcounter{tocdepth}{2}
\tableofcontents

\section{Introduction}\label{s: intro}

In this paper, we study systems of polynomial equations in different classes of rings and algebras. For each $R$ in one of these classes  we  interpret by systems of equations a ring of integers $O$ of a number field or a  global function field (i.e.\ $O$ is the integral closure of $\mbb Z$ or $\mbb F_p[t]$ in a finite extension of $\mbb Q$ or $\mbb F_p(t)$, respectively).  In particular, this reduces the Diophantine problem (decidability of systems of polynomial equations) in $O$ to the same problem in $R$. Hence if $\mc{D}(O)$ is undecidable, then also $\mc{D}(R)$ is undecidable. It is known that $\mc{D}(O)$ is undecidable if  $O$ has positive characteristic \cite{Shla_vertical_def_functions}, and it is conjectured to be also undecidable if otherwise $O$ is a ring of algebraic integers \cite{Phe_Zah, Denef_conjecture}. 

A \emph{number field} is a finite field extension of $\mbb Q$. A \emph{global function field} is a finite extension of $\mbb F_p(t)$, for some prime $p$. A ring of integers of a number field  is called a \emph{ring of algebraic integers}.

The Diophantine problem in a structure $R$, denoted $\mc{D}(R)$, asks whether there exists an algorithm that, given a \emph{system} of equations $S$ with coefficients in $R$, determines if $S$ has a solution in $R$ or not. 
The original modern version  of the Diophantine problem (also called Hilbert's Tenth Problem or generalized Hilbert's Tenth Problem) was posed by Hilbert for the ring of integers $\mbb Z$. This was solved in the negative in 1970 by Matiyasevich \cite{mat} building on the work of Davis, Putnam, and Robinson \cite{DPR}. Subsequently the same problem has been studied in a wide variety of rings, most notably in $\mbb Q$ and in rings of algebraic integers $O$,
where it remains widely open. As mentioned above, a long-standing conjecture   \cite{Denef_conjecture, Phe_Zah} states that $\mbb Z$ is Diophantine in any such $O$  (and thus  $\mc{D}(O)$ is undecidable). This conjecture  has been verified in some particular cases \cite{Sha_Shla, Shla_book, new}, and it has been shown to be true assuming the Safarevich-Tate conjecture \cite{Mazur2010}.

The situation is much clearer for rings of integers of global function fields, i.e.\ for finite field extensions of rational function fields of the form $\mbb{F}_p(t)$ for some variable $t$ and  some prime integer $p$. Indeed, Shlapentokh \cite{ Shla_holomorphy_positive_char} showed that $\mbb{F}_p[t]$ is Diophantine in any such ring $O$, and consequently that $\mc{D}(O)$ is undecidable.

 Some commutative rings where the Diophantine problem remains open are most remarkably $\mbb Q$ (it is known however that this problem is undecidable in $\mbb Z[S^{-1}]$, for $S$ an infinite set of primes of Dirichlet density 1 \cite{Poonen_on_Q}); the rational functions $\mbb C(t)$  (even though $\mc{D}(\mbb C (t_1,t_2))$ is undecidable \cite{Kim_Roush}); and the field of  Laurent series $\mbb{F}_p((t))$.  We refer to  \cite{Poonen, Phe_Zah, Shla_book, Koenigsmann2014} for further information and surveys of results in this direction.

Eisentraeger \cite[Theorem 7.1]{phd_eisentrager} proved the general result that for any finitely generated associative commutative unitary ring $R$, the Diophantine problem $\mc{D}(R)$ is undecidable conditionally on the conjecture that $\mc{D}(O)$ is undecidable for any ring of algebraic integers. Moreover, she showed that $\mc{D}(R)$ is undecidable in many cases, see \cite[Theorem 7.1]{phd_eisentrager} or Theorem \ref{t: phd_eisentrager} in this paper. 

 Regarding non-commutative rings, Romankov \cite{Romankov_eqns_2} showed that $\mc{D}(F)$ is undecidable in several types of free rings $F$, which include free Lie rings, free associative or non-associative rings, and free nilpotent rings.  One can view these rings as free $\mbb{Z}$-algebras, it is essentail, since the proofs use  undecidability of the Diophantine problem in the coefficients $\mbb{Z}$.  Using different methods Kharlampovich and Miasnikov  recently proved undecidability of $\mc{D}(A)$, for any of the following rings $A$: a free associative $k$-algebra, a free Lie $k$-algebra (of rank at least $3$), and group $k$-algebras $k(G)$ for various groups $G$ (including free, torsion-free hyperbolic, right-angled Artin, and other groups)  \cite{KM_free_algebras, KM_free_Lie_algebras}. In all these results the field $k$ is arbitrary, possibly with decidable $\mc{D}(k)$.  \\

\noindent We proceed to describe the results obtained in the present work.   In this paper we convene that \emph{all rings and algebras are possibly non-associative, non-commutative, and non-unitary} unless stated otherwise. A ring or algebra $R$ is called \emph{unitary} if and only if it has a multiplicative identity. Algebras will over be considered over associtative commutative unitary ring, and we fix $\L$ to denote such ring. Given a $\L$-algebra $L$, we let $L^2$ be the $\L$-module generated by all products of two elements of $L$.  
In this paper the notion of ring is equivalent to the notion of $\mbb Z$-algebra.

he main tool used in this paper is the so-called \emph{interpretability  by systems of equations} (or \emph{e-interpretability}), which is a variation of the classical notion of the first-order interpretability, where instead of arbitrary first-order formulas, finite systems of equations are used as the interpreting formulas (see Definition \ref{interpDfn} for details).  The main relevant property of such interpretations is that if $A_1$ is e-interpretable in $A_2$ then $\mc{D}(A_1)$  is Karp-reducible to $\mc{D}(A_2)$ by a polynomial time many-one reduction (Karp reductions). All reductions mentioned in this paper are of this type.

In number theoretic terms, an  \emph{interpretation by systems of equations}  is roughly a Diophantine definition  up to a Diophantine equivalence relation. Here Diophantine definitions are considered using systems of equations, as opposed to single equations.
We  convene that all systems of equations and all e-interpretations allow the use of any constant elements of the structures at hand, not necessarily in the signature.  See  Subsections \ref{s: e_interpretations}, \ref{Dioph_pblms_intro} and \ref{p: notation} for further comments on these matters.

One of the main results of the paper is the following. By $\Lr$ we refer to the language of rings with constants. We write $(R;\mc{L})$ to indicate that a structure $R$ is considered with a language $\mc{L}$.

\begin{theorem}\label{t: 1_intro} 
Let $A$ be a ring  (possibly non-associative, non-commutative, and non-unitary). Assume that $A$  is finitely generated as an abelian group, and that $ A^2$ is infinite.  Then there exists a  ring of algebraic integers $O$ such that $(O;\Lr)$  is e-interpretable (see below) in $(A; \Lr)$, and $\mc{D}(O; \Lr)$ is Karp-reducible to $\mc{D}(A,\Lr)$. If otherwise $A^2$ is finite, 
then $\mc{D}(R;\Lr)$ is decidable. 
\end{theorem}

Theorem \ref{t: 1_intro} is further generalized  to algebras.
The language of $\L$-modules $\La$, or of  $\L$-algebras $\Lal$, consists in the usual language of groups $\Lg$ or of rings $\Lr$, respectively,  together with unary functions $\{\cdot_\l \mid \l\in \L\}$ representing multiplication by elements of $\L$ (see Subsection \ref{p: notation}).

\begin{theorem}\label{t: 2_intro}
Let $R$ be a (possibly non-associative, non-commutative, and non-unitary)   algebra over a finitely generated associative commutative unitary ring $\L$. Suppose  that $R$ is finitely generated as a $\L$-module.  Then if,  $ R^2 = \langle \{xy\mid x\in R, y\in R\} \rangle_\L$ is infinite,  there exists a  ring of  integers $O$ of a number field or a  global function field such that $(O;\Lr)$  is e-interpretable in $(R; \Lal)$, and the Diophantine problem $\mc{D}(O;\Lr)$ is Karp-reducible to $\mc{D}(R;\Lal)$. Moreover: 
\begin{enumerate}
\item If $R^2$ is infinite and $\L$ has positive characteristic, then $(\mbb{F}_p[t];\Lr)$ is e-interpretable  in $(R;\Lal)$ for some prime integer $p$, and $\mc{D}(R;\Lal)$ is undecidable. 
\item If  $R^2 $ is finite and $\mc{D}(R; \La)$ is decidable, then $\mc{D}(R;\Lal)$ is decidable. 
\end{enumerate}
If $\L$ is a finite field, then all the above holds after replacing $(R;\Lal)$ by $(R;\Lr)$.  
\end{theorem}

Theorems \ref{t: 1_intro} and \ref{t: 2_intro} are further extended to other classes of finitely generated rings and algebras, including associative commutative non-unitary rings. We say that a non-unitary ring $R$ has characteristic $n\in\mbb{N}$ if $n$ is the smallest nonnegative integer such that $nr = 0$ for all $r\in R$.

\begin{theorem}\label{t: intro_non_unit_rings}
Let $A$  be a finitely generated associative commutative non-unitary ring, with   $A^2$   infinite. Then there exists a ring of integers $O$ of a number or a global function field such that $(O;\Lr)$ is e-interpretable in $(A; \Lr)$, and $\mc{D}(O;\Lr)$ is Karp-reducible to $\mc{D}(A;\Lr)$. 

Moreover, if $A$ has positive characteristic, then the following holds: $O$ is the ring of integers of a global function field; the ring of polynomials  $(\mbb{F}_p[t]; \Lr)$ is e-interpretable in $A$ for  prime integer $p$; and $\mc{D}(A; \Lr)$ is undecidable.
\end{theorem}

In \cite[Theorem 7.1]{phd_eisentrager}, Eisentraeger  studied the Diophantine problem in finitely generated associative commutative unitary rings. The main result of that work is stated in this paper in Theorem \ref{t: phd_eisentrager}.  The above Theorem \ref{t: intro_non_unit_rings}, together with the aforementioned result,  provide insight on the Diophantine problem in finitely generated associative commutative rings, unitary or not. Indeed, for any such ring $R$,   one can reduce $\mc{D}(O;\Lr)$  to $\mc{D}(R;\Lr)$  for some ring of integers $O$ of a number or global function field, and in a wide variety of cases $O$ turns out to be a ring of integers of a global function field, making $\mc{D}(R; \Lr)$ undecidable due to Shlapentokh's work \cite{ Shla_holomorphy_positive_char}.

We also prove a  statement analogous to Theorem \ref{t: intro_non_unit_rings} for algebras:

\begin{theorem}\label{t: intro_non_unit_algebras}
 Let $L$ be a finitely generated associative commutative non-unitary  algebra over a finitely generated associative commutative unitary ring $\Theta$, with $L^2$ infinite. Then there exists a  ring of  integers $O$ of a number field or a  global function field such that $(O;\Lr)$  is e-interpretable in $(L; \Lal)$, and $\mc{D}(O;\Lr)$ is Karp-reducible to $\mc{D}(L;\Lal)$. 

Moreover,  if $\Theta$ has positive characteristic, then $O$ is the ring of integers of a global function field,  $(\mbb{F}_p[t];\Lr)$ is e-interpretable  in $(L;\Lal)$ for some prime integer $p$, and $\mc{D}(L;\Lal)$ is undecidable.  
\end{theorem}

Our results also involve some classes of possibly  non-associative and non-commutative algebras and rings. We only give the statement for algebras, keeping in mind that the statement for rings is obtained by taking $\L = \mbb{Z}$. We need the following definition:  Let $L$ be a $\L$-algebra, and let $T$ be a generating set of $R$. If $R$ is non-unitary then we let $I_n(T)$ or $I_n$ denote the $\L$-ideal generated by all  products of $n$ elements of $T$. If $L$ is unitary then we let $I_n(T)$, or $I_n$ in short, denote the $\L$-ideal generated by all  products of $n$ elements of $T\setminus\{\lambda\cdot 1 \mid \lambda\in \Lambda\}$, where $1$ denotes the multiplicative identity of $R$.  We say that $L$ is \emph{left-normed-generated} with respect to $T$ if  for all $n \geq 1$, $I_n(T)$ is generated as a $\L$-module by a (possibly infinite) set of elements of the form $(t_1(t_2( \dots (t_{k-1} t_k)\dots)))$, with $k\geq n$ and $t_i \in T$ for all $i$.

\begin{theorem}\label{t: intro_non_commutative}
Let $L$ be a   finitely generated  algebra (possibly  non-associative, non-commutative and non-unitary) over a finitely generated associative commutative unitary ring $\L$. Suppose  that $L$ is  left-normed-generated with respect to some finite generating set $T$, and that $(L/I_{n}(T))^2$ is infinite for some $n\geq 1$. Then there exists a  ring of  integers $O$ of a number field or of a  global function field such that $(O;\Lr)$ is e-interpretable in $(R; \Lal)$, and $\mc{D}(O; \Lr)$ is  Karp-reducible to $\mc{D}(L; \Lal)$. Moreover:
\begin{enumerate}

\item If $\L$ has positive characteristic $p$, then $(\mbb{F}_p[t];\Lr)$ is e-interpretable in $(L; \Lal)$, and $\mc{D}(L;\Lal)$ is undecidable.

\item  If $L$ is a ring (i.e.\ $\L=\mbb{Z}$) then $O$ is a ring of algebraic integers.
    
\end{enumerate}
If $\L$ is $\mbb Z$ or a finite field then all the above holds after replacing $(L;\Lal)$ by $(L;\Lr)$.
\end{theorem}

We obtain the following applications of the result above. By $[R/I_n,R/I_n]$ we denote the $\L$-submodule of $R/I_n$ generated by   $\{[x,y]\mid x,y\in R/I_n\}$. 

\begin{corollary}\label{t: intro_Lie}
Let $L$ be a  finitely generated  Lie $\L$-algebra. Assume that $[R/I_n,R/I_n]$ is infinite for some $n\geq 1$, and that $\L$ is finitely generated.  Then the conclusions of Theorem \ref{t: finitely_generated_algebras}  hold for $L$.
\end{corollary}

\begin{corollary}\label{c: intro_free}
Let $F$ be a finitely generated  free associative $\L$-algebra (possibly non-commutative and non-unitary) or a free Lie algebra of rank at least $2$, with $\L$ finitely generated. Then the conclusions of Theorem \ref{t: finitely_generated_algebras} hold for $F$.
\end{corollary}

This complements the previously mentioned results of Romankov \cite{Romankov_eqns_2} and of Kharlampovich and Miasnikov \cite{ KM_free_algebras, KM_free_Lie_algebras} regarding free algebras. We remark that in \cite{Romankov_eqns_2} it is proved  (among others) that the algebras of Corollary \ref{c: free} actually have undecidable Diophantine problem if $\L=\mbb Z$.

It is known that the first order theory of any ring of integers of a number field or of a global function field is undecidable. Therefore we obtain the following consequence:

\begin{theorem}\label{t: 4_intro}
Suppose that $A$ satisfies the hypotheses of any of the theorems and corollaries above. Then the first-order theory of $A$ in the corresponding language with constants is undecidable.
\end{theorem}

The above result extends Noskov's work \cite{Noskov_rings} where it is proved that all finitely generated infinite associative commutative unitary rings have undecidable first-order theory.

\paragraph{From bilinear maps to associative commutative unitary algebras.}

 The main techniques developed in this paper allow to move from studying arbitrary rings and algebras (possibly non-associative,  non-commutative, and non-unitary) to  studying finitely generated associative commutative unitary rings. The Diophantine problem in the latter  scenario is more or less understood, modulo the Diophantine problem of  rings of integers of  number fields, as was shown in  \cite[Theorem 7.1]{phd_eisentrager}.  

The  reduction from arbitrary rings   (algebras) to associative commutative unitary rings (algebras) is achieved through the study of rings of scalars of bilinear maps between $\L$-modules, where $\L$ is a finitely generated associative commutative unitary ring (when dealing with rings we have $\L=\mbb{Z}$). These are relevant for us because much of the structure of a  $\L$-algebra (or a ring) can be ``seen'' in its ring multiplication operation, which is indeed a $\L$-bilinear map between $\L$-modules.   In fact, bilinear maps also arise naturally in other structures, and in some cases, it is possible to apply the methods presented in this paper  to these, for example, in some classes of groups. In  \cite{GMO_solvable} we explore further this line of work for several classes of solvable groups.
 Next we describe further our approach with bilinear maps.  Some of the ideas we present now were introduced by the second named author in \cite{Myasnikov1990}, and they have been used successfully to study different first-order theoretic aspects of different types of structures, including rings whose additive group is finitely generated \cite{Mi_So_4}, free algebras \cite{KM3, KM2, KM1}, and nilpotent groups \cite{  Mi_so_3, Mi_So_2}. Our contribution is a treatment of these ideas by means of systems of equations.

Observe that ring multiplication $\cdot$ of a $\L$-algebra $R$ is, by definition, a $\L$-bilinear map between $\L$-modules. One can try to replace $\L$  by a ``larger'' associative commutative unitary ring $\Delta$.  To do so, one needs to find a ring $\Delta$ that acts on $R$ by  $\L$-module endomorphisms (thus making the additive group of $R$ into a  $\Delta$-module), in a way that $\cdot$ becomes a $\Delta$-bilinear map between $\Delta$-modules. In this case, we say that $\Delta$ is a \emph{ring of scalars} of the multiplication map $\cdot$.  

These considerations apply in the same way if one starts with an arbitrary $\L$-bilinear map $f:N\times N \to M$ between $\L$-modules $N$ and $M$. If $f$ is full and non-degenerate (see Subsection \ref{s: rs_full_nonfegenerate_bilin}) then one can define the \emph{largest} ring of scalars of $f$, denoted $R(f)$. This ring constitutes an important feature of $f$, and in some sense, it provides an ``approximation'' to interpreting (in $(N,M;f;\La)$) multiplication of constant elements from $N$ and $M$ by integer variables, or by variables taking values in $\L$. Another important property of $R(f)$ is that it is interpretable in $(N,M;f;\La)$ by first-order formulas without constants \cite{Myasnikov1990}.  In this paper, we prove that this is still true if one uses systems of equations instead (with constants).

\begin{theorem}\label{10_intro}
If $f$ is full and non-degenerate, and if  $N$ and $M$ are finitely generated, then both $Z(\Sym(f))$ and the largest ring of scalars $R(f)$ of $f$ are e-interpretable in the two sorted structure $(N,M; f, \La)$. 
\end{theorem}

 By $\End(N)$ we denote   the algebra  of $\L$-endomorphisms of $N$. The ring $\Sym(f)$ is defined as  $$\Sym(f)=\{\alpha \in \End(A)\mid f(\alpha x,y)=f(x,\alpha y) \ \forall \ x,y\in A\},$$ and $Z(\Sym(f))$ denotes the center of $\Sym(f)$ (i.e.\ the set of elements from $\Sym(f)$ that commute with all elements from $\Sym(f)$). The interest we have for $Z(\Sym(f))$ is mostly technical. This is explained in Remark \ref{r: choose_ring}.

We next provide an idea of the proof of Theorem \ref{10_intro}. There are two main observations. The first goes as follows: Both $Z(\Sym(f))$ and $R(f)$ can be seen as subalgebras $\End(N)$. Let $a_1, \dots, a_k$ be a  module generating set of $N$.  Then each $\alpha \in \End(N)$ can be identified with the tuple $(\alpha a_1, \dots, \alpha a_k) \in N^k$, and so we can think of $Z(\Sym(f))$ and $R(f)$ as $\L$-submodules of $N^k$ with an extra ring multiplication operation. In particular we manage to first e-interpret the whole algebra $\End(N)$ in $(N;\Lal)$, together with the action of any element of $\End(N)$ on the elements $\{a_1, \dots, a_k\}$. 

The second idea is to use the properties of $f$ in order to ``express'' statements about  endomorphisms from $Z(\Sym(f))$ and of $R(f)$ in terms of their actions on $a_1, \dots, a_k$. For example, given $\alpha, \beta, \gamma\in Z(\Sym(f))$, one has  that $\gamma=\alpha \beta$ if and only if 
$
f(\gamma a_i, a_j)= f(\beta a_i, \alpha a_j)$ for all  $1\leq i,j\leq k
$
(this is proved using bilinearity of $f$ and the fact that $f(\alpha \beta x,y)=f(\beta x, \alpha y)$ for all $x$ and $y$). This and the considerations in the previous paragraph can be combined to  show (after some work) that  multiplication in $Z(\Sym(f))$ is e-interpretable in $(N,M;f,\La)$.  The rest of the proof follows in a similar fashion, with the e-interpretation of $R(f)$ being more involved but of a likewise spirit.\\

In Subsection \ref{s: arbitrary_bilinear_maps} we generalize Theorem \ref{10_intro} to the following result. 
\begin{theorem}\label{t: 11_intro}
Let $f:A \times B\to C$ be a $\Lambda$-bilinear map between finitely generated $\L$-modules. Then there exists an   associative commutative unitary ring $\Theta$   that is a $\L$-algebra and is finitely generated as a $\L$-module, with the property that $(\Theta;\Lal)$ is  e-interpretable  in $F=(A,B,C;f, \La)$. If $\L$ is the ring $\mbb Z$  or a field, then $\Lal$ and $\La$ can be replaced by $\Lr$ and $\Lg$, respectively.
\end{theorem}
As mentioned above, the ring multiplication of any    $\L$-algebra $R$, finitely generated as a $\L$-module, is a $\L$-bilinear map $\cdot: R\times R \to R$ between finitely generated $\L$-modules, and $(R,R,R;\cdot, \La)$ is  e-interpretable in $(R;\Lal)$. Applying Theorem \ref{t: 11_intro} and transitivity of e-interpretations we manage to  move from the possibly non-associative, non-commutative, and non-unitary $R$ to an associative, commutative, unitary algebra.

As we discussed in the abstract of the present paper, we believe that the above Theorem \ref{t: 11_intro} constitutes one of the main contributions of the paper, with potential applicability to other structures other than rings or algebras.

\section{Preliminaries}

\subsection{Interpretability by systems of equations}

\subsubsection{Multi-sorted structures} 

A \emph{multi-sorted structure} $\mathcal{A}$ is a tuple $\mathcal{A}=(A_i; f_j, r_k, c_{\ell} \mid i\in I,j\in J,k\in K,\ell\in L),$ where the $A_i$ are pairwise disjoint sets called \emph{sorts}; the $f_j$ are functions of the form
$
f_j: A_{i_{1}} \times \dots \times A_{i_{m}} \to A_{i_{m+1}} 
$ for some $i_{1}, \dots, i_{m+1}\in I$;
the $r_k$ are relations of the form
$
r_k: A_{i_1'} \times \dots \times A_{i_{p}'} \to \{0,1\}, 
$
for some $i_1', \dots, i_p'\in I$; and the $c_{\ell}$ are constants, each one belonging to some sort.  The tuple $(f_j, r_k, c_\ell \mid j\in J,k\in K,\ell\in L)$ is called the \emph{signature} or the \emph{language} of $\mc{A}$. We always assume   that $\mc{A}$ contains the relations "equality in $A_i$", for all sorts $A_i$, but we do not write them in the signature. 

 All our terms will allow the use of any constant element in any sort, regardless of whether the constant is in the signature of the structure at hand. For this reason, and without losing generality, we will  always work with (multi-sorted) structures without constants in the signature.

If $\mc{A}$ has only one sort then $\mc{A}$ is a structure in the usual sense.  One can construct terms in a multi-sorted structure in an analogous way as in uniquely-sorted structures.  In this case, when introducing a variable $x$, one must specify a sort where it takes values, which we denote $A_x$.

%

%

A set $S$ of \emph{generators} of $\mc{A}$ is a collection of elements from different sorts such that any element from any sort can be written as a term using only constants from $S$ and from the signature of $\mc{A}$ (and using function symbols).

\begin{example}
A ring is a structure $(R, +, \cdot, =)$ with one sort $R$, the operations of addition $+$ and multiplication $\cdot$, and the equality relation $=$. When there is no risk of ambiguity we will always denote a classical one-sorted structure, such as a ring  or a group, simply by its sort, that is we denote a ring $(R, +, \cdot, =)$ simply by $R$.

In this paper we understand left modules over a ring as one-sorted structures in the way explained in Section \ref{p: notation}. An alternative formulation arises by considering two sorts $A$ (the underlying abelian group) and $R$ (the ring acting on $A$), a group addition $+_A$,  ring addition and multiplication $+_R, \cdot_R$,  equality relations $=_R$, $=_A$, and an action operation $\cdot: R\times A \to A$. We stress again that this is not the approach taken in this paper.
\end{example}

Let $\mathcal{A}_1, \dots, \mathcal{A}_n$ be a collection of multi-sorted structures. 
We let $(\mathcal{A}_1, \dots,  \mathcal{A}_n)$ be the multi-sorted structure that is formed by all the sorts, functions, relations, and constants of each $\mathcal{A}_i$. Given a function $f$ or a relation $r$ we  use the notation $(\mc{A},f)$ or $(\mc{A},r)$ to denote the multi-sorted structure $\mc{A}$ with the additional function $f$ or relation $r$. 

\begin{example}
The following example will be used later in the paper.  Let $A, B, C$ be abelian groups, and let $f:A\times B \to C$ be a bilinear map, i.e.\ a map such that for all $a\in A$ the map $f(a, \cdot): B \to C$ is a group homomorphism, and similarly, for all $b\in B$ the map $f(\cdot, b): B \to C$ is a group homomorphism.

Then one can consider the multi-sorted structure $(A, B, C, f)$. This is formed by the sorts  $A, B, C$; the group operations and relations of $A, B$, and $C$, and the operation given by the map $f$.

An example of a terms in  $(A, B, C, f)$ is $f(x,b) + y$ where $b$ is an element from $B$, and $x,y$ are variables taking values in $A$ and $C$, respectively.  
\end{example}

\subsubsection{Diophantine problems and reductions.}\label{Dioph_pblms_intro} 
Let $\mathcal{A}$ be  a multi-sorted structure. An \emph{equation in $\mc{A}$}  is an expression  of the form $r(\tau_1, \dots, \tau_k)$, where $r$ is a signature relation of $\mc{A}$ (typically, the equality relation), and each $\tau_i$ is a term in $\mc{A}$ (taking values in an appropriate sort) where some of its variables may have been substituted by elements of $\mc{A}$. Such elements are called the \emph{coefficients} (or the \emph{constants}) of the equation. These may not be signature constants. 
A system of equations is a finite conjunction of equations. A \emph{solution} to a system of equations $\wedge_i \Sigma_i(x_1, \dots, x_n)$ on variables $x_1, \dots, x_n$ is a tuple $(a_1, \dots, a_n)\in A_{x_1}\times \dots \times A_{x_{n}}$ such that all equations $\Sigma_i(a_1, \dots, a_n)$ are true in $\mc{A}$.

The \emph{Diophantine problem} in $\mc{A}$, denoted  $\mc{D}(\mc{A})$, refers to the algorithmic problem of  determining if each given system of equations in $\mc{A}$ (with coefficients in a fixed computable set) has a  solution. Sometimes this is  also called \emph{Hilbert's tenth problem} in $\mc{A}$. An algorithm $L$ is  a \emph{solution} to $\mc{D}(\mc{A})$ if, given a system of equations $S$ in $\mc{A}$, determines whether $S$ has a solution or not. If such an algorithm exists, then $\mc{D}(\mc{A})$ is called \emph{decidable}, and, if it does not, \emph{undecidable}.

An algorithmic problem $P_1$ is said to be \emph{Karp-reducible} (or \emph{polynomial-time many-one reducible}) to another problem $P_2$ if there is a polynomial-time algorithm that transforms inputs to problem $P_1$ into inputs to problem $P_2$, such that both problems have the same output given an input and the transformed input, respectively. 

A crucial observation is that if $P_1$ is undecidable, and $P_1$ is Karp-reducible to $P_2$, then $P_2$ is undecidable as well. 

In some cases, one restricts the set of coefficients $C$ that can be used in the input equations of the Diophantine problem of a structure. For instance, one typically takes $C=\mbb Z$ when studying  $\mc{D}(\mbb Q)$ (equivalently one can take $C=\{0,1\}$). In this paper, we will always need that $C$ contains certain coefficients, namely those used in a certain e-interpretation, and maybe also the preimage of some constants of the structure that is being e-interpreted. For this reason, and to simplify the exposition, we agree that $C$ is always the whole structure, or a suitable computable subset if the structure is not countable.

\subsubsection{Interpretations by systems of equations}\label{s: e_interpretations}
In this section we review the notion of interpretability by systems of equations between multi-sorted structures. Here we use a much more general setting since our arguments will require handling a variety  of multi-sorted structures.

\emph{Interpretability by systems of equations (e-interpretability)} is the analog of the classic model-theoretic notion of interpretability by first-order formulas (see \cite{Hodges, Marker}). In e-interpretability one requires that only systems of equations \emph{with coefficients} are used, instead of first-order formulas.  From a number theoretic viewpoint, e-interpretability is roughly Diophantine definability by systems of equations up to a Diophantine definable equivalence relation. 

In this paper \----in e-interpretations and Diophantine problems\---- we consider \emph{systems} of equations and not just single equations. This may contrast with some number-theoretic settings, where systems of equations are equivalent to single equations, and both notions are treated interchangeably, i.e.\ when studying  integral domains whose field of fractions is not algebraically closed.

\begin{definition}
Let $\mc{A}$ be a structure with sorts $\{A_i \mid i\in I\}$. A \emph{basic set} of $\mc{A}$ is a set of the form
$
A_{i_1} \times \dots \times A_{i_m}
$
for some $m$ and $i_j$'s. 
\end{definition}

\begin{definition} 
Let $M$ be a basic set of a multi-sorted structure $\mathcal{M}$. A subset $A\subset M$ is called \emph{definable by equations} (or \emph{e-definable}) in $\mathcal{M}$  if there exists a system of equations
$
\Sigma_A(x_1,\ldots,x_m, y_1, \dots, y_k)$ on variables $(x_1, \dots, x_m, y_1, \dots, y_k)=(\mathbf{x}, \mathbf{y})$ such that $\mathbf{x}$ takes values in $M$, and such that 
for any tuple $\mathbf{a}\in M$, one has that  $\mathbf{a} \in A$ if and only if the system $\Sigma_A(\mathbf{a}, \mb{y})$ on variables $\mathbf {y}$  has a solution in $\mathcal{M}$. In this case $\Sigma_A$ is said to \emph{define} $A$ in $\mc{M}$.  The integer $n$ is called the \emph{dimension} of the e-definition.
\end{definition}

From an algebraic geometric viewpoint, an e-definable set is a projection onto some coordinates of an affine algebraic set. From a number theoretic point, it is a Diophantine definable set, allowing to use systems of equations rather than a single equation.

\begin{definition}\label{interpDfn}
Let $\mathcal{A}= \left(A_1, \dots; f, \dots, r \dots\right)$  and $\mc{M}$ be two multi-sorted structures. One says that $\mc{A}$ is  \emph{interpretable by equations} (or \emph{e-interpretable})  in   $\mathcal{M}$ if for each sort $A_i$ there exists a basic set $M(A_i)$ of $\mathcal{M}$, a subset $X_i\subseteq M_i$,  and a surjective map 
$
\phi_i:X_i \to A_i
$ 
such that:
\begin{enumerate}
\item $X_i$ is e-definable in $\mathcal{M}$, for all $i$. 
\item For each function $f$ and each relation $r$ in the signature of $\mc{A}$ (including the equality relation of each sort), the preimage by $\boldsymbol{\phi}=(\phi_{1}, \dots)$  of the graph of $f$ (and of $r$) is e-definable in $\mc{M}$, in which case we say that $f$ (or $r$) is e-interpretable in $\mc{M}$. The same terminology applies to functions and relations that are not necessarily in the signature of $\mc{A}$.   
\end{enumerate} 
The tuple of maps $\boldsymbol{\phi}=\left(\phi_1, \dots\right)$ is called  an \emph{e-interpretation} of $\mathcal{A}$ in $\mathcal{M}$.  
The map $\phi$ is called  an \emph{e-interpretation} of $R_1$ in $R_2$. 
We will say that $\mc{A}$ is \emph{e-interpretable} in $\mc{M}$ if there exists an e-interpretation $\phi$ of $\mc{A}$ in $\mc{M}$. It is usually clear, but not important, what the specific e-interpretation is. 
\end{definition}

The next lemma illustrates a key application of e-interpretability.

\begin{lemma}\label{factor_is_interpretable}
Let $R$ be a ring, not necessarily commutative or associative. Suppose $I\leq R$ is an ideal that admits a $1$-dimensional e-definition in $R$. Then $R/I$  is e-interpretable in $R$.
\end{lemma}
\begin{proof}
Let $\Sigma_I(x, y_1, \dots, y_m)$ be a system of equations giving a 1-dimensional e-definition of $I$  in $R$, so that $a \in R$ belongs to $I$ if and only if $\Sigma_I(a,  y_1, \dots, y_m)$ has a solution on $y_1, \dots, y_m$. It suffices to check that the natural epimorphism $\pi:R\to R/I$ is an e-interpretation of $R/I$ in $R$. First observe that the preimage of $\pi$ is the whole $R$, which is clearly e-definable in $R$. Regarding the preimage of the equality relation of $R/I$, we have that  $\pi(a_1) =\pi(a_2)$ in $R/I$  if and only if $a_1-a_2\in I$, i.e.\ if and only if $\Sigma_I(a_1-a_2,  y_1, \dots, y_m)$ has a solution. From this it follows that the preimage  of equality in $R/I$,  i.e.\
$
\left\{ a_1, a_2 \in R \mid \pi(a_1) = \pi(a_2) \right\},
$
is e-definable in $R$ by the system of equations $\Sigma_I'(x_1, x_2, y_1, \dots, y_m)$ obtained from $\Sigma_I(x,  y_1, \dots, y_m)$ after substituting each occurrence of $x$ by $x_1-x_2$, where $x_1$ and $x_2$ are fresh new variables.  

By similar arguments, the preimages of the addition and multiplication operations of $R/I$ are e-definable in $R$: indeed, for any three elements $a_1, a_2, a_3\in R$ we have that $\pi(a_1)+\pi(a_2) = \pi(a_3)$ if and only if $a_1+a_2-a_3 \in I$, and $\pi(a_1)\pi(a_2) = \pi(a_3)$ if and only if $a_1a_2-a_3\in I$.
\end{proof}

Interestingly, all finitely generated ideals of a ring are e-interpretable in it:

\begin{lemma}\label{interpretations_in_rings}
Let $I$ be a finitely generated ideal of a ring $R$. Then $I$ is e-definable in $R$. As a consequence, $R/I$ is e-interpretable in $R$.
\end{lemma}

\begin{proof}
Let $a_1, \dots, a_n$ be a generating set of $I$. Then the equation $x=\sum x_i a_i$ on variables $(x, x_1, \dots, x_n)$ e-defines $I$ in $R$.  Lemma \ref{factor_is_interpretable} now implies that  $R/I$ is e-interpretable in $R$.
\end{proof}

Note that any finitely generated ring $R$ (as all rings considered in this work) is Noetherian, i.e. all their ideals are finitely generated (this follows from Hilbert's basis theorem). Thus any ideal $I$ of $R$ is e-definable in $R$ and $R/I$  is e-interpretable in $R$.

\begin{remark}\label{r: factor_is_interpretable}

It is clear from the proof that an analog of Lemma \ref{factor_is_interpretable} holds for other structures, such as groups with e-definable normal subgroups,  modules with e-definable submodules, etc.
\end{remark}

The following remark will be used several times without referring to it.

\begin{remark}\label{r: contained_e_def_implies_e_interp}
Let $\mc{A}=(A; f,\dots; r,\dots)$ and $\mc{B}=(B; f',\dots; r',\dots)$ be uniquely-sorted structures such that $A \subseteq B$ and all functions and relations of $\mc{A}$ are also functions and relations of $\mc{B}$. Assume $A$ is e-definable in $\mc{B}$. Then, clearly, $\mc{A}$ is e-interpretable in $\mc{B}$.
\end{remark}

The next two results are fundamental. They follow from Lemma \ref{equation_reduction}, which we present at the end of this subsection.

\begin{proposition}[\normalfont{E-interpretability is transitive}]\label{interpretation_transitivity}
If $\mc{A}$ is e-interpretable in $\mc{B}$ and $\mc{B}$ is e-interpretable in $\mc{M}$, then $\mc{A}$ is e-interpretable in $\mc{M}$.
\end{proposition}

\begin{proposition}[\normalfont{Reduction of Diophantine problems}]\label{Diophantine_reduction}
Let $\mc{A}$ and $\mc{M}$ be (possibly multi-sorted) finitely generated structures such that $\mc{A}$  is e-interpretable in $\mc{M}$.  Then $\mc{D}(\mc{A})$  is Karp-reducible to $\mc{D}(\mc{M})$. As a consequence, if  $\mc{D}(\mc{A})$ is undecidable, then so is $\mc{D}(\mc{M})$.  

Similarly, the first-order theory of $\mc{A}$  is  Karp-reducible to the first-order theory (with constants\footnote{The considerations made regarding the use of constants in systems of equations and Diophantine problems apply as well for first-order formulas and their decidability problems (see Paragraph 3 of Subsection \ref{p: notation} or Subsection \ref{Dioph_pblms_intro}).}) of $\mc{M}$, and the second is undecidable if the first is.
\end{proposition}

Both Propositions \ref{interpretation_transitivity} and  \ref{Diophantine_reduction} are consequences of the following lemma, which states in technical terms that if one structure is e-interpretable in the other, then one may ``express'' equations in the first as systems of equations in the second. Similar results to this with analogous proof are well-known. We include a proof for completeness.
\begin{lemma}\label{equation_reduction}
Let $\boldsymbol{\phi}=(\phi_1, \dots)$  be an e-interpretation of a multi-sorted structure $\mathcal{A}=(A_1, \dots; f,\dots, r, \dots)$ in another multi-sorted structure $\mc{M}$,
with $\phi_{i}: X_i \subset M(A_i)\to A_i$ (see Definition \ref{interpDfn}). Let $\sigma(\mb{x})=\sigma(x_1, \dots, x_n)$ be an arbitrary system of equations in $\mc{A}$ with each variable $x_i$ taking values in $A_{j_i}$. Then there exists a system of equations $\Sigma_\sigma(\mb{y}_1, \dots, \mb{y}_{n})$ in $\mc{M}$, such that each tuple of variables $\mb{y}_i$ takes values in $M(A_{j_i})$, and  such that a tuple $(\mb{b}_1, \dots, \mb{b}_n)\in \prod_{i=1}^n M(A_{j_i})$ is a solution to  $\Sigma_\sigma(\mb{y}_1,\ldots,\mb{y}_n)$ if and only if $(\mb{b}_1, \dots, \mb{b}_n)\in \prod_{i=1}^n X_{j_i}$ and $(\phi_{j_1}(\mb{b}_1), \dots, \phi_{j_n}(\mb{b}_n))$ is  a solution to $\sigma$.
\end{lemma}
\begin{proof}
 We claim that, by introducing new variables and new equations, we can rewrite $\sigma$ so that $\sigma$ consists in a conjunction of equations $\sigma_1\wedge\dots \wedge \sigma_m$ such that the following holds: For all $i=1, \dots, m$, $\sigma_i$ is either of the form $z=f(x_1, \dotsm x_n)$,  $r(x_1,\dots, x_n)$, or $z=a$, where $f$ is some function from $\mc{A}$, the symbol $r$ is some relation from $\mc{A}$, the symbols $x_1,\dots, x_n, z$ are variables, and $a$ is any element from the sorts of $\mc{A}$.  The lemma follows from the claim, since by the definition of e-interpretability, the present lemma is true for each of $\sigma_i$, $i=1, \dots, m$. Hence, it suffices to take $\Sigma_\sigma$ to  be  $\Sigma_{\sigma_1}\wedge\dots\wedge\Sigma_{\sigma_n}$.

 We now prove the claim. We proceed by induction on the syntactic length $|\sigma|$ of $\sigma$, the base cases being clear. We can assume that $\sigma$ consists on a single equation. Suppose first that $\sigma$ does not have the desired form and that it is of the form $z=f(t_1,\dots, t_m)$ for some variable $z$, some function $f$, and some terms $t_1,\dots, t_{m}$ depending on some variables $x_1, \dots, x_k$.  We can rewrite $\sigma$ into the equivalent system of equations $z=f(y_{1},\dots, y_{m}) \wedge y_1=t_1(x_1,\dots, x_k) \wedge \dots \wedge y_{m}=t_{m}(x_1, \dots, x_k)$, where $y_1, \dots, y_m$ are new variables. The syntactic length of each one of these equations is strictly less than $|\sigma|$, and then we can proceed by induction.  If $\sigma$ has the form $r(t_1, \dots, t_m)$ for some relation $r$ and some terms $t_1, \dots, t_m$, we can proceed similarly by rewriting it as $r(y_1, \dots, y_m)\wedge\wedge y_1=t_1(x_1,\dots, x_k) \wedge \dots \wedge y_{m}=t_{m}(x_1, \dots, x_k)$. Finally if $\sigma$ has  the  form $f(t_1, \dots, t_m) = g(t_1', \dots, t_{k}')$ for some functions $f,g$ and terms $t_1, \dots, t_m, t_1', \dots, t_k'$, we can  rewrite sigma in the form $z = f(t_1, \dots, t_m) \wedge z= g(t_1', \dots, t_{k}')$. Each one of the equations in the conjunction has syntactic length smaller than $|\sigma|$, and again we can proceed by induction. This proves the claim.
\end{proof}

We will use the following observations in different occasions:

\begin{remark}\label{r: obserbation_interpretations}
Let $\mc{A}$, $\mc{B}$, and $\mc{M}$ be  (possibly multi-sorted) structures. Suppose that all sorts among the sorts of $\mc{A}$ and $\mc{B}$ are pairwise disjoint. Let $\mc{N}$ be a (possibly multi-sorted) structure which is the result of adding functions, relations, constants, or more sorts to $\mc{M}$.  Suppose $\mc{A}$ is e-interpretable in $\mc{M}$.  Then $\mc{A}$ is also e-interpretable in $\mc{N}$.

Moreover, if both $\mc{A}$ and $\mc{B}$ are e-interpretable in $\mc{M}$, then the multi-sorted structure $(\mc{A}, \mc{B})$  is also e-interpretable in $\mc{M}$.
\end{remark}

\subsection{The Diophantine problem in finitely generated associative commutative unitary rings}

 Recall that, given two associative commutative unitary rings $R, S$ with $S\subseteq R$, and an element $r\in R$, we say that $r$ is \emph{integral} over $S$ if there exists a monic polynomial $p(x)\in S[x]$ such that $p(r)=0$.  The \emph{integral closure} of $S$ in $R$ is defined as the subset of integral elements of $R$, and it forms a subring of $R$.  We will often denote it $O_R$   when the ring $S$ is understood from the context, and we will refer to $O_R$ as the \emph{ring of integers of $R$.}

A \emph{number field} is a finite field extension of $\mbb{Q}$. Given a field $k$, by $k(t)$ we denote the set of rational functions with coefficients in $k$ and variable $t$. The integral closure of $\mbb{Z}$ in a number field $K$  is called a \emph{ring of algebraic integers}. 
 A \emph{global function field} is a finite field extension of $\mbb{F}_p(t)$ for some finite field $\mbb{F}_p$ with a prime number $p$ of elements. 
 
Shlapentokh proved that the Diophantine problem is undecidable in any ring of integers of a global function field (see 10.6.2 from \cite{Shla_book}).  On the other hand, it is conjectured that the same is true for any ring of algebraic integers.

In the PhD thesis \cite[Theorem 7.1]{phd_eisentrager} Eisentraeger proved the  following relation between  the Diophantine problem of any infinite finitely generated commutative associative unitary ring, and the Diophantine problem of rings of integers of number or global function fields. Recall that the \emph{characteristic} of a ring with unity is the minimum positive integer $n$ such that $1 + \overset{n}{\dots} + 1 = 0$.

\begin{theorem}[Theorem 7.1 in \cite{phd_eisentrager}]\label{t: phd_eisentrager}
Let $A$ be an infinite finitely generated associative commutative unitary ring.
\begin{enumerate}
\item Assume that the characteristic of $A$ is n $>0$. If $A$ has infinitely many elements, then
Hilbert’s Tenth Problem for A is undecidable.
\item  Assume $A$ has characteristic zero. If the Krull dimension of $A$ is at least $2$, then $\mc{D}(A;\Lr)$ is undecidable. If we assume that $\mc{D}(O_K;\Lr)$ is undecidable for any ring of algebraic integers $O_K$ of integers of number fields is undecidable, then $\mc{D}(A;\Lr)$ is undecidable.
\end{enumerate}
\end{theorem}

We are interested in a variation, slightly stronger, formulation of this result, which we state below. The proof of this alternative formulation can be obtained from the proof of Theorem \ref{t: phd_eisentrager} \cite{phd_eisentrager} by making straightforward modifications. For completeness, we include an alternative yet similar  proof in the appendix at the end of this paper.

We   need the notion of \emph{rank} of an abelian group, and by extension of a ring. This is defined in a variety of manners throughout the literature. Here  we   follow \cite{Fuchs}.   
\begin{definition}[\cite{Fuchs}]\label{l: trank}
The \emph{rank} of an abelian group $A$ is  the maximal number of nonzero elements $a_1,\ldots,a_n\in A$ such that whenever $\alpha_1 a_1 + \dots + \alpha_n a_i=0$ for some integers $\alpha_1, \dots, \alpha_n\in \mbb{Z}$, then $\alpha_i a_i=0$ for all $i=1, \dots, n$. 

The \emph{rank} of a ring is defined as the rank of $R$ seen as an abelian group (i.e.\ forgetting its multiplication operation). 
\end{definition}

\begin{remark}\label{r: trank_domains}
 Let $R$ be an integral domain. If $R$ has zero characteristic then the rank of $R$ coincides with the \emph{dimension} of $R$ seen as a $\mbb{Z}$-module, which is the maximum number of $\mbb{Z}$-linearly independent elements in $R$, i.e.\ elements $a_1, \dots, a_r$ such that whenever $\sum_{i=1}^r \alpha_i a_i=0$ for some integers $\alpha_i$, we have $\alpha_i=0$ for all $i=1, \dots, r$. If $R$ has positive characteristic $p>0$, then the rank of $R$ is the dimension of $R$ as a $\mbb{F}_p$-vector space.
\end{remark}

Hence the notion of rank generalizes   dimension of $\mbb{Z}$-modules and of vector spaces. As an  example we have that the rank of the non-integral domain  $R=\mbb{Z}[x]/(px)$ is infinite. However, note that  $R$   has only one linearly independent element over $\mbb{Z}$, hence $R$ seen as a $\mbb{Z}$-module has dimension $1$. On the other hand, $R$ does not admit the structure of a vector space over a finite field. Another illustrative example is given by the integral domain $\mbb{Z}[\frac{1}{2}]$, which has rank $1$.

\begin{theorem}\label{t: f_g_commutative_rings}

Let $R$ be an infinite finitely generated commutative ring with identity.
Then there exists a ring of integers $O$ of a number or a global function field such that $(O;\Lr)$ is e-interpretable in $(R;\Lr)$,  and $\mc{D}(O;\Lr)$ is Karp-reducible to $\mc{D}(R;\Lr)$. Moreover, one of the following holds:

\begin{enumerate}
    \item If $R$ has positive characteristic $p> 0$, then the following holds: $O$ is the ring of integers of a global function field; the ring of polynomials  $(\mbb{F}_p[t];\Lr)$ is e-interpretable in $(R;\Lr)$ for some variable $t$; and $\mc{D}(R;\Lr)$ is undecidable.
    \item  If  $R$ has zero characteristic and it has infinite rank then the same conclusions as above hold:   $O$ is the ring of integers of a global function field; the ring of polynomials  $(\mbb{F}_p[t];\Lr)$ is e-interpretable in $(R;\Lr)$ for some prime $p$ and variable $t$; and $\mc{D}(R;\Lr)$ is undecidable. 
    \item  If  $R$ has zero characteristic and it has finite rank $n$, then $O$ is a ring of algebraic integers, and $\mc{D}(R;\Lr)$ is undecidable provided that $\mc{D}(O;\Lr)$ is undecidable. Additionally, $K$ is a field extension of $\mbb{Q}$ of degree at most $n$. 
\end{enumerate}
\end{theorem}

\begin{proof}
See Appendix \ref{appendix}.
\end{proof}

\subsection{Notation and conventions}\label{p: notation}

Here we note and emphasize some relevant aspects of the notation used in the paper.


\paragraph{1.} Unless stated otherwise, all rings and algebras are not necessarily associative, commutative, or  unitary. 

Given an  algebra $R$ over a ring $\L$, and a subset $S\subseteq R$, we let     $\langle S\rangle_\L$ be  the left $\L$-submodule of $R$ generated by a set $S$. We also let $R^2=\langle  xy\mid x,y\in R\rangle_\L$.

All modules are assumed to be \emph{left} modules over commutative associative unitary rings. Similarly, the underlying module of an algebra is assumed to be a left module over commutative associative unitary rings. All arguments work in the same way if we replace all left modules by right modules or all left module by bimodules.

\paragraph{3.} The language of additive groups is $\Lg=(+)$. The language of  $\L$-modules is $\La=(\Lg, \cdot \L)$, where the $\cdot \L=\{\cdot\l\mid \l \in \L\}$ are unary functions representing multiplication by scalars: $\cdot_\l(x)=\l x$. The language of rings  $\Lr$ is  $(+,\cdot)$. The language of $\L$-algebras is $\Lal=(+,\cdot,\cdot\L)$. If $\L$  admits a finite generating set $S$, then one can replace $\cdot \L$ by $\cdot S= \{\cdot_\lambda \mid \l \in S\}$ without loss of generality.

Hence, in an equation (or in a formula) over a $\L$-module or $\L$-algebra $R$, one is allowed to multiply any element of $R$ by any constant element of $\L$. But this is as far as one can involve $\L$: no variable can take values in $\L$, no quantification over $\L$ can be made, etc.

\paragraph{4.} The notion of $\mbb{Z}$-module or  $\mbb{Z}$-algebra with the languages above is equivalent, for the purposes of studying decidability of the Diophantine problem, to the notion of abelian group or ring, respectively.

\paragraph{5.}   Sometimes we will want to look at an algebra $L$ over a ring $\L$ \emph{as a $\L$-module}, or \emph{as a ring}, or \emph{as a group}, forgetting about the corresponding additional operations of $L$. We will use the notation $(L;\La)$, $(L;\Lr)$, $(L;\Lg)$ when this is done, respectively.   We will also write $(L;\Lal)$ to emphasize that $L$ is considered with all its $\L$-algebra operations. A similar terminology will be used for other structures such as rings and modules.

This notation will be used extensively in expressions of the type \emph{$(L;\mc{L}_1)$ is e-interpretable in $(K;\mc{L}_2)$}. This  means that the structure $L$ with  the operations of the language $\mc{L}_1$  is e-interpretable in $K$ considered with the operations of $\mc{L}_2$. In the particular case that $\mc{L}_1=\mc{L}_2$ we will also say that \emph{$L$ is e-interpretable in $K$ in the language $\mc{L}_1$}.

\section{From bilinear maps to commutative rings and algebras}\label{s: bilinear_maps}

A brief description of the arguments used in this section can be found in the last part of the introduction.

\subsection{Ring of scalars of a full non-degenerate bilinear map}\label{s: rs_full_nonfegenerate_bilin}

Throughout this subsection, $\L$ denotes an associative, commutative, unitary ring, possibly infinitely generated.

A map $f: N\times N \to M$ between   $\Lambda$-modules $N$ and $M$ is \emph{$\Lambda$-bilinear} if, for all $a\in N$, the maps $\ell_a: N\to M$ and $r_a: N \to M$ defined as $\ell_a(b)=f(a,b)$ and $r_a(b)=f(b, a)$ are  homomorphisms of $\Lambda$-modules. 
We call $f$ \emph{non-degenerate} if whenever $f(a,x)=0$ for all $x\in N$, we have $a=0$, and   also whenever $f(x,a)=0$ for all $x\in N$ we have $a=0$. The map  $f$ is called \emph{full} if the $\Lambda$-submodule generated by the image of $f$ is the whole $M$.

The set of module endomorphisms of a  $\Lambda$-module $N$, denoted $\End(N)$, forms an associative unitary $\Lambda$-algebra once we equip it with the operations of addition and composition (henceforth called \emph{multiplication}). Given  $\alpha\in \End(N)$ and $x\in N$,  we   write $\alpha x$ instead of $\alpha(x)$. An action of a ring  $\Delta$ on $N$ is  a ring homomorphism $\phi: \Delta \to \End(N)$. Any such action $\phi$ endows $N$ with a structure of $\Delta$-module. The action is called \emph{faithful} if $\phi$ is an embedding.

\begin{definition}
Let $\L$ be a commutative associative unitary ring, let $N$ and $M$  be $\Lambda$-modules, and let $f:N \times N \to M$ be a $\Lambda$-bilinear map between $N$ and $M$. A ring  $\D$  is called a \emph{ring of scalars}  of $f$ if it is associative, commutative, and unitary, and there exist faithful actions of $\D$ on $M$ and $N$ such that $f(\alpha x,y)=f(x,\alpha y)=\alpha f(x,y)$ for all $\alpha\in \Delta $ and all  $x,y\in N$.
\end{definition}

Since the actions of a ring of scalars  $\D$ of $f$ on $M$ and $N$ are faithful, there exist ring embeddings $\D\hookrightarrow \End(M)$ and $\D\hookrightarrow \End(N)$. For this reason and for convenience, we  always assume that a ring of scalars of $f$ is a subring of $\End(N)$. 
\begin{definition}\label{d: Rf}
We say that $\D$ is the \emph{largest} ring of scalars of $f$ if for any other ring of scalars $\D'$ of $f$, one has $\D'\leq \D$ as subrings of $\End(N)$. We denote such ring by $R(f)$.
\end{definition}

We will also need the following notation:
\begin{definition}\label{d: sym_and_Zsym}
Let $\L$ be an associative commutative unitary ring, and let $N$ be a $\L$-module. We define the following subsets of $\End(N)$:
\begin{align}
\Sym(f)&=\left\{ \alpha\in \End(N) \mid f(\alpha x,y)=f(x,\alpha y) \ \hbox{for all} \ x,y\in N\right\}, \label{e: Symf}\\ 
Z(\Sym(f))&=\left\{\alpha\in \Sym(f) \mid \alpha\beta=\beta\alpha\ \hbox{for all} \ \beta\in \Sym(f)\right\}.\label{e: ZSymf}
\end{align}
It is straightforward to check that both $\Sym(f)$ and $Z(\Sym(f))$ are $\L$-modules. 
\end{definition}

The next result was proved by the second author in \cite{Myasnikov1990}. We recover its proof since we will need to elaborate on it in the next subsection.

\begin{theorem}[\normalfont{cf.\ \cite{Myasnikov1990}}]\label{maxringexists}
Let $\L$ be an associative commutative unitary ring, and let $f:N\times N\to M$ be a full non-degenerate bilinear map between  $\L$-modules. Then the largest ring of scalars $R(f)$ of $f$  exists and is unique.
\end{theorem}
\begin{proof}
First observe that for all $\alpha_1,\alpha_2 \in Z(\Sym(f))$  and all $x,y\in N$,
$$
f(\alpha_1\alpha_2 x,y)=f(\alpha_2 x,\alpha_1 y)=f(x,\alpha_2\alpha_1 y)=f(x,\alpha_1\alpha_2 y),
$$
and thus $\alpha_1\alpha_2 \in \Sym(f)$. Since  both  $\alpha_1$ and $\alpha_2$ commute with any element from $\Sym(f)$, so does $\alpha_1 \alpha_2$. Hence, $\alpha_1\alpha_2\in Z(\Sym(f))$,  and so $Z(\Sym(f))$ is a $\L$-subalgebra of $\End(N)$.

Next, let $\D$ be an arbitrary ring of scalars of $f$. We will show $\D$ is a subring of $Z(\Sym(f))$. Indeed, by definition, $\D \subseteq \Sym(f)$. To see that $\D \subset Z(\Sym(f))$, let $\alpha\in \D$ and $\beta\in \Sym(f)$. Then, for all $x,y \in N$,
$$
f(\alpha\beta x,y)=\alpha f(\beta x,y)=\alpha f(x,\beta y)=f(\alpha x,\beta y)=f(\beta \alpha x,y).
$$ 
Hence
$
f((\alpha \beta-\beta \alpha)x,y)=0
$ for all $x, y \in N$.
Since $f$ is non-degenerate and $y$ is arbitrary, $(\alpha\beta-\beta \alpha)x=0$ for all $x\in N$. It follows that $\alpha\beta=\beta \alpha$,  and thus $\D\subseteq Z(\Sym(f))$.

By what we have seen so far, $Z(\Sym(f))$ is an associative commutative unitary $\L$-algebra that acts faithfully on $N$. We now wish to find a subring of $Z(\Sym(f))$, call it $\Theta$, that acts on $M$. Since $f$ is full, for all $z\in M$ we have $z=\sum_i f(x_i,y_i)$ for some $x_i, y_i \in N$. Hence, one may try to define the following action:
\begin{equation}\label{action_eqn}
\alpha z=\sum f(\alpha x_i,y_i) \quad \hbox{for} \quad \alpha \in \D.
\end{equation} 
However, this is not necessarily well-defined, because the same $z\in M$ may have different expressions as sums of elements $f(x_i, y_i)$. With this in mind, we let $\Theta$ be the set of all $\alpha\in Z(\Sym(f))$ such that
\begin{equation}\label{well_dfn_eqn}
\sum f(\alpha x_i,y_i) =\sum f(\alpha x_i', y_i') \quad \hbox{whenever} \quad \sum f(x_i,y_i)=\sum f(x_i',y_i').
\end{equation}
Clearly, $\Theta$ is closed under addition and multiplication, and therefore it is a subring of $Z(\Sym(f))$ with a well-defined action on $M$ given by \eqref{action_eqn}. Since the action of $\Theta$ on $N$ is faithful, and $f$ is a non-degenerate map, the  action of $\Theta$ on $M$ is faithful as well. It follows that $\Theta$ is a ring of scalars of $f$. Moreover, any other ring of scalars $\D$ of $f$ satisfies \eqref{well_dfn_eqn}, and thus, since $\D\subset Z(\Sym(f))$ by our previous argument, we have $\D \leq \Theta$. We conclude that $\Theta$ is the unique largest ring of scalars of $f$. 
\end{proof}

\begin{remark}
It is clear from the proof above that $R(f)$ is closed under multiplication by $\L$. Hence, $R(f)$ admits the structure of a $\L$-algebra. 
\end{remark}

\subsection{E-interpreting $Z(\Sym(f))$ and  the largest ring of scalars}

Throughout this subsection    $\L$ denotes a \emph{Noetherian} associative commutative unitary ring (possibly infinitely generated).  Recall that a ring is \emph{Noetherian} if for every infinite ascending chain of ideals $I_1\subseteq I_2 \subseteq \dots$ there exists $n$ such that $I_{n} = I_{m}$ for all $m\geq n$.  In this case, any finitely generated $\L$-module is Noetherian and finitely presented (see \cite{Eisenbud} or \cite{Goodearl_Warfield}). We refer to Subsection \ref{p: notation} for important notation and terminology conventions.

The goal of this subsection is to prove the following result.

\begin{theorem}\label{maxringdefinable}
Let $f: N \times N \to M$ be a full non-degenerate bilinear map between  finitely generated $\Lambda$-modules. Then both $Z(\Sym(f))$ (see Definition \ref{d: sym_and_Zsym}) and the largest ring of scalars $R(f)$ of $f$ are  finitely generated as $\L$-modules, and they are e-interpretable as $\L$-algebras in $F=(N,M;f, \La)$. Moreover,
\begin{enumerate}

\item $Z(\Sym(f))$, $R(f)$, $N$, and $M$ are all simultaneously   finite, or they are all simultaneously infinite.

\item If $\L$ is a field, then $(Z(\Sym(f));\Lr)$  is e-interpretable in $(N,M;f, \Lg)$ (i.e.\  multiplication by scalars is not required).

\item If $\L=\mbb{Z}$, then both $(Z(\Sym(f));\Lr)$ and $(R(f);\Lr)$ are e-interpretable in $(N,M;f, \Lg)$.

\end{enumerate}

\end{theorem}

We state some lemmas and observations before proving Theorem \ref{maxringdefinable}, starting with a useful description of $\End(N)$.

\begin{remark}\label{r: Ends_are_tuples}
Let $N$ be a  $\Lambda$-module with finite module presentation $\langle a_1, \dots, a_m \mid \sum_i x_{j,i} a_i$, $j=1,\ldots, T\rangle_\L$,   where $x_{j,i}\in \Lambda$ for all $i,j$.   Each element $\alpha$ of $\textit{End}_{\L}(N)$ uniquely determines an $m$-tuple $(\alpha a_1,\ldots, \alpha a_{m})\in N^{m}$, and one has $\sum_i x_{j,i} (\alpha a_i)=0$ for all $j$. Conversely, any $m$-tuple from $N^m$ with this last property determines an element from $\End(N)$.  Thus  $\End(N)$ can be identified with the set of $m$-tuples $(\alpha_1, \dots, \alpha_m) \in N^m$ that satisfy $\sum_i x_{j,i} \alpha_i=0$ for all $j$.  

In the particular case that $\L$ is a field we have that $N$ is a vector space. In particular, $N$ is a free $\L$-module, and so it admits a finite presentation with an empty set of relations. In this case, $\End(N)=N^{m'}$ for some $m' \leq m$.  Let us mention a particular case when $N$ is a free $\L$-module. In this case $N$ admits a finite presentation with an empty set of relations. In this case, $\End(N)=N^{m'}$ for some $m'\leq m$. This happens, for example, if $\L$ is a field or if $\L=\mbb{Z}$ and $N$ is torsion-free. 
\end{remark}

The above identification of $\End(N)$ with a subset of $N^m$ is used to prove the following result.

\begin{lemma}\label{End_is_interpretable}
Let $N$ be a finitely generated $\L$-module. Then the following hold:
\begin{enumerate}
\item $(\End(N); \La)$ is e-interpretable in $(N;\La)$.  

\item Let  $S_N=\{a_1, \dots, a_m\}$ be a generating set of $N$, and define maps $\cdot a_i: \End(N) \to N$ so that $\cdot a_i$ sends each $\alpha\in \End(N)$ to  $\alpha a_i \in N$. Denote $\cdot S_N= \{\cdot a_1, \dots, \cdot a_m\}$.  Then the two-sorted structure  $\textit{END}_{\L}(N)=\left( \End(N), N;  \cdot S_N, \La \right)$ is e-interpretable in $(N;\La)$.   

\item  In the particular case that $\L$ is a field or the ring of integers $\mbb{Z}$,  the previous statements are still valid after replacing $\La$ by $\mc{L}_{\it group}$ in all structures.
\end{enumerate}
\end{lemma}

\begin{proof}
As mentioned above, since $\Lambda$ is a Noetherian associative commutative unitary ring, any finitely generated $\Lambda$-module is finitely presented with respect to any finite generating set. Let $\sum x_{j,i} a_i$, $j=1,\ldots, T$  be a finite set of relations of $N$, with $x_{j,i}\in \Lambda$ for all $i,j$.  

Following Remark \ref{r: Ends_are_tuples},  identify each element $\alpha$ of $\End(N)$ with the $m$-tuple $(\alpha_1,\ldots,\alpha_{m})=(\alpha a_1,\ldots, \alpha a_{m})\in N^{m}$. By this same remark, any $m$-tuple $\alpha=(\alpha_1,\ldots,\alpha_{m})\in N^m$ belongs to $\End(N)$ if and only if $\sum x_{j,i} \alpha_i=0$ for all $j$. This is a finite system of equations in $(N;\La)$ with variables $\alpha_i$, and so $\End(N)$ as a set  is e-definable  in $(N; \La)$. As observed in Remark \ref{r: Ends_are_tuples}, if $\L$ is a field then $\End(N)=N^{m'}$ for some $m'\leq m$, and so the e-definition consists in an empty equation. In particular, it does not use multiplication by scalars. Hence $\End(N)$ is e-definable as a set in $(N;\Lg)$. 

The group addition of two tuples from the $\L$-module $\End(N)$ is obtained  by component-wise addition. Hence the graph of the addition operation of $\End(N)$ (which is a subset of $N^{3m}$) is e-definable in $(N, \La)$. By similar reasons, so are the graphs of the equality relation of  $\End(N)$ and  of multiplication by fixed elements of $\Lambda$ (i.e.\ multiplication by scalars). This proves that  $(\End(N); \La)$ is e-interpretable in $(N; \La)$. In the case that $\L$ is a field,  $(\End(N);\Lg)$ is e-interpretable  in $(N; \mc{L}_{\it group})$.  

It follows that the two-sorted structure $\left(\End(N), N; \La \right)$ is e-interpretable in $(N; \La)$. Finally, notice that, for $\alpha = (\alpha_1, \dots, \alpha_m) \in \End(N)$ and $x\in N$, the tuple $(\alpha, x) =  N^{m} \times N$ belongs to the graph of $\cdot a_i$ if and only if $x= \alpha a_i = \alpha_i$. In other words, for any tuple $(y_1, \dots, y_{m+1}) \in N^{m+1}$ we have
$$
 (y_1, \dots, y_{m+1}) \in {\it Graph}(\cdot a_i) \subseteq N^{m+1} \quad \hbox{if and only if} \quad y_i=y_{m+1},
$$  
hence the graph of $\cdot a_i$ is e-definable in $(N; \Lg)$. This completes the proof that $\textit{END}_{\L}(N)$ is e-interpretable in $(N; \La)$. 

If $\L$ is a field then multiplication by scalars was not used in any equation other than when  e-interpreting the scalar multiplication of $\End(N)$. If $\L=\mbb{Z}$ then a $\L$-module is just a group, because  $nx= x+\overset{n}{\dots} + x$ for all $n\in \mbb Z$. Hence, Item 3 holds. 
\end{proof}

\begin{remark}\label{r: F_1}
%
It follows from Lemma \ref{End_is_interpretable} and Remark \ref{r: obserbation_interpretations} that there exists an e-interpretation $\boldsymbol{\phi}$ of the three-sorted structure 
$$
F_1=\left(\End(N), N, M;  f, \cdot S_N, \La \right)
$$ 
in $F=\left(N, M; f, \La \right)$.  If $\L$ is a field or $\mbb{Z}$, then one can replace $\La$ by $\Lg$.

Thus by transitivity of e-interpretations (Proposition \ref{interpretation_transitivity}), in order to prove that $(R(f); \Lal)$ or $(Z(\Sym(f)); \Lal)$ is  e-interpretable in $F$ it suffices to show  that it e-interpretable in $F_1$. For this one must keep in mind that an equation in $F_1$ can involve constants and variables from  $N$, $M$, and $\End(N)$; the  map $f$;  actions of endomorphisms on the $a_i$'s given by $\cdot S_N$; and the operations of $(N;\La)$, $(M;\La)$, and $(\End(N);\La)$ without its ring multiplication. For example, the equation $f(\alpha a_i,a_j)=f(a_i,\alpha a_j)$ on the variable $\alpha$  is valid in $F_1$, whereas $\alpha_1\alpha_2 a_i=\alpha_2\alpha_1  a_i$ or $\alpha x=a_i$  is not (for variables $\alpha_1, \alpha_2, \alpha \in \End(N)$, $x\in N$). 
\end{remark}

We next prove the main result of this subsection.

\begin{proof}[Proof of Theorem \ref{maxringdefinable}]
First observe that     $\End(N)$ is finitely generated as a $\L$-module, because $N^m$ is a Noetherian module and $\End(N)$  embeds as a $\L$-module into $N^m$, by Remark \ref{r: Ends_are_tuples}. By the same reason both $R(f)$ and $Z(\Sym(f))$  are finitely generated as  $\Lambda$-modules.  

Denote $F=(N,M;f,\La)$. We  proceed to prove that $(Z(\Sym(f));\Lal)$ is e-interpretable     in $F$.   By the previous Remark \ref{r: F_1},  it suffices to show that $(Z(\Sym(f));\Lal)$ is e-interpretable in $F_1$ for some generating set $S_N=\{a_1, \dots, a_n\}$ of $N$. 

We start by proving that $\Sym(f)$  can be e-defined as a subset of $\End(N)$ in $F_1$. Indeed, take any $x,y \in N$  and write $x=\sum x_ia_i$ and $y=\sum y_i a_i$ for some $x_i, y_i \in \Lambda$. Since $\alpha x = \sum x_i \alpha a_i$ for all $\alpha \in \End(N)$, we have  $f(\alpha x,y)=\sum x_i y_j f(\alpha a_i, a_j)$, and similarly for $f(x,\alpha y)$. It follows that
\begin{equation}\label{e: 2020}
\Sym(f)=\left\{ \alpha\in \End(N) \mid f(\alpha a_i,a_j)=f(a_i,\alpha a_j) \ \hbox{for all} \ 1\leq i, j \leq n \right\}.
\end{equation} 
Observe that $\alpha a_i=\cdot a_i(\alpha)$ for all $i=1, \dots, n$. Hence \eqref{e: 2020} can be written as a first-order sentence in $F_1$.  We conclude that $\Sym(f)$ is e-definable as a set in $F_1$  by the system of equations
\begin{equation}\label{Sym_dfn_system}
\bigwedge_{1\leq i, j \leq n} \big[ f\left(\cdot a_i(\alpha), a_j\right) = f(a_i, \cdot a_j(\alpha)) \big]
\end{equation}
on the single variable $\alpha$ taking values in  $\End(N)$. 
Note that \eqref{Sym_dfn_system} does not use multiplication by scalars $\L$.

Since the signature of $F_1$ contains all operations of    $(\End(N);\La)$, we have that   $\Sym(f)\leq \End(N)$ as a $\L$-module is e-interpretable in $F_1$.

Next, we show that $Z(\Sym(f))$  is e-definable as a set in $F_1$. As before,  this  immediately implies that the $\L$-module $(Z(\Sym(f)); \La)$ is e-interpretable in $F_1$. Let $\beta_1,\ldots,\beta_k$ be a finite generating set of $(\Sym(f); \La)$. Then, $\alpha\in Z(\Sym(f))$ if and only if $\alpha\in \Sym(f)$ and  $\alpha \beta_t=\beta_t \alpha$ for all $t=1, \dots, k$. This implies that 
\begin{equation}\label{aux_eqn}
f(\alpha a_i, \beta_t a_j) = f(a_i, \alpha \beta_t a_j) = f(a_i, \beta_t \alpha a_j) = f(\beta_t a_i, \alpha a_j) \quad \hbox{for all}\ i, j, t.
\end{equation}
We claim that \eqref{aux_eqn} is a sufficient condition for an endomorphism $\alpha \in \Sym(f)$ to belong to $Z(\Sym(f))$. As a consequence one has that $Z(\Sym(f))$ is definable as a set in $F_1$  by means of the following system  of equations on the variable $\alpha$:  
\begin{equation}\label{Z_Sym_f_definable}
 \bigwedge_{\substack{t=1,\dots, k,\\ 1\leq i,j\leq n}} \big[f(\cdot a_i(\alpha), \cdot a_j(\beta_t))=f(\cdot a_i(\beta_t), \cdot a_j(\alpha)) \big],
\end{equation}
together with the system \eqref{Sym_dfn_system}, which ensures that $\alpha\in \Sym(f)$. Again, recall that $\alpha a_i$ and $\beta_t a_j$ are written in $F_1$ in the form $\cdot a_i (\alpha)$ and $\cdot a_j(\beta_t)$. As before, \eqref{Z_Sym_f_definable} does not use multiplication by the scalars $\L$. 

To prove the claim, i.e.\ that \eqref{aux_eqn} is a sufficient condition for $\alpha\in \Sym(f)$ to belong to $Z(\Sym(f))$,  suppose \eqref{Z_Sym_f_definable} holds. Then
$
f(\beta_t\alpha a_i, a_j) = f(\alpha \beta_t a_i, a_j)
$
for all $i,j,t$, and thus, for fixed $i$ and $t$, $f([\beta_t, \alpha]a_i, a_j)=0$ for all $j$, where  $[\beta_t, \alpha] = \beta_t \alpha  - \alpha \beta_t$.  By bilinearity of $f$ and the fact that $a_1, \dots, a_n$ generate $N$, we have that $f([\beta_t, \alpha]a_i, x)=0$ for all $x\in N$ and for all $i, t$. Since $f$ is non-degenerate, $[\beta_t, \alpha]a_i =0$ for all $i, t$. This implies that $[\beta_t, \alpha]x=0$ for all $x \in N$, and thus $[\beta_t, \alpha]=0$  for all $t$. This completes the proof of the claim.

We have  seen that  the $\L$-module $(Z(\Sym(f));\La)$ is e-interpretable in $F_1$. Moreover, multiplication by scalars $\L$ is only used for defining the the scalar multiplication of $(Z(\Sym(f));\La)$. Hence in fact $Z(\Sym(f))$ as a group is e-interpretable in $\left(\End(N), N, M;  f, \cdot S_N, \Lg \right)$ (i.e.\ $F_1$ after replacing $\La$ by $\Lg$).

By analogous reasons as above, for any triple $\gamma_1, \gamma_2, \gamma_3 \in Z(\Sym(f))$ the equality $\gamma_3 = \gamma_1 \gamma_2$ holds if and only if
\begin{equation}\label{dfn_mult}
f(\gamma_3 a_i, a_j) =  f(\gamma_2 a_i, \gamma_1 a_j) \quad \hbox{for all}\ 1\leq i,j\leq n.
\end{equation}
Hence the ring multiplication of $Z(\Sym(f))$ is e-interpretable in $F_1$ by means of  \eqref{dfn_mult} and appropriate systems of the form \eqref{Sym_dfn_system} and \eqref{Z_Sym_f_definable} (which ensure that $\gamma_i \in Z(\Sym(f))$). We conclude that $Z(\Sym(f))$ as a $\L$-algebra is e-interpretable in $F_1$, and hence in $F=(N,M;f,\La)$.

We now prove Items 2 and 3 of the statement of the Theorem \ref{maxringdefinable}. As observed in the arguments above, multiplication by scalars of $F_1$ was only used in order to e-interpret  multiplication by scalars of $Z(\Sym(f))$. Hence  $(Z(\Sym(f));\Lr)$ is e-interpretable  in $\left(\End(N), N, M;  f, \cdot S_N, \Lg \right)$. By Lemma \ref{End_is_interpretable} and Remark \ref{r: F_1}, if $\L$ is either $\mbb{Z}$ or a field, then the latter structure is e-interpretable in $(M, M; f, \Lg)$. Hence  $(Z(\Sym(f));\Lr)$ is e-interpretable as a ring in $(M,M; f, \Lg)$. This concludes the proof of Items 2 and 3.
Next we show that $(R(f);\Lal)$ is e-interpretable in $(N,M;f,\La)$.
By the previous  arguments and by transitivity of e-interpretations, it suffices to prove  that $(R(f);\Lal)$  is e-interpretable in $F_1$. First recall from the proof of Theorem \ref{maxringexists} that $R(f)$ is the set of all $\alpha \in Z(\Sym(f))$ such that 
$$
\hbox{if}\ \ \sum_{i=1}^t f(x_i,y_i)=\sum_i f(x_i',y_i'),\ \ \hbox{then} \ \ \sum_{i=1}^{t} f(\alpha x_i,y_i)=\sum_i f(\alpha x_i',y_i'),
$$ 
for any $t\geq 1$  and elements $x_1, \dots, x_t, y_1, \dots, y_t, x_1', \dots, x_t', y_1,\dots, y_t'$ in $N^t$.
This condition is equivalent to
\begin{equation}\label{Af_final_condition}
\hbox{if} \  \ \sum_{i=1}^t f(x_i,y_i)=0, \ \ \hbox{then} \ \ \sum_{i=1}^t f(\alpha x_i,y_i)=0.
\end{equation}
We claim that $\alpha \in Z(\Sym(f))$ satisfies \eqref{Af_final_condition} if and only if it satisfies the following condition:  
\begin{equation}\label{Af_condition_2}
\hbox{if} \ \  \sum_{1\leq j,k \leq n} z_{j,k} f(a_j,a_k)=0 \ \ \hbox{for some} \ \ z_{j,k}\in \Lambda \ (1\leq j,k\leq n), \ \ \hbox{then} \ \ \sum_{1\leq j,k \leq n} z_{j,k} f(\alpha a_j,a_k)=0.
\end{equation} 
Indeed, the direct implication is immediate. Conversely, suppose  that $\alpha$ satisfies \eqref{Af_condition_2}, and let $t\geq 1$ and $x_1, \dots, x_t, y_1, \dots, y_t \in N$ be such that $\sum_{i=1}^t f(x_i, y_i)=0$. Write each $x_i$ and $y_i$ in terms of the generators  $a_1, \dots, a_n$, 
$$
x_i=\sum_{j=1}^n x_{i,j} a_j,\quad \hbox{and} \quad y_{i=1}^n=\sum_{k} y_{i,k} a_{k}, \quad x_{i,j}, y_{i,k}\in \Lambda.
$$ 
Since $f$ is bilinear, 
$$
\sum_{i=1}^t f(x_i,y_i)=\sum_{1\leq j,k \leq n} \left(\sum_{i=1}^t x_{i,j}y_{i,k}\right) f(a_j,a_{k}) =0.
$$
Thus by \eqref{Af_condition_2}, $$0=\sum_{1\leq j,k\leq n} \sum_{i=1}^t x_{i,j}y_{i,k}f(\alpha a_j, a_k)= \sum_{i=1}^t f(\alpha x_i, y_j).$$ This completes the proof of the claim.

The set $S$ of all tuples $(z_{i,j}) \in \L^{n^2}$ such that $\sum_{1\leq i,j\leq n} z_{i,j}f(a_i,a_j)=0$ forms a submodule of $\Lambda^{n^2}$, and so it admits a finite generating set, say $X=\{\mb{s}_i\mid i=1,\ldots,T\}$. Write  $\mb{s}_i=(s_{i,j,k} \mid 1\leq j,k\leq n)$. 
Then $\alpha \in Z(\Sym(f))$ belongs to $R(f)$ if and only if 
\begin{equation}
\sum_{1\leq j,k\leq n} \left( \sum_{i=1}^t q_i s_{i,j,k} \right) f(\alpha a_j, a_k)  = 0, \quad \hbox{for all} \quad q_1, \dots, q_t \in \Lambda.
\end{equation}
Equivalently,
\begin{equation}\label{aux_eq_3}
\sum_{i=1}^t q_i \left(\sum_{1\leq j,k\leq n} s_{i,j,k} f(\alpha a_j, a_k) \right) = 0, \quad \hbox{for all} \quad q_1, \dots, q_t \in \Lambda.
\end{equation}
By making appropriate choices for the $q_i$'s, one sees that \eqref{aux_eq_3} holds if and only if each one of the expressions inside the parenthesis is $0$. It follows that $R(f)$ is e-definable as a set  in $F_1$. Consequently, $R(f)$ is e-interpretable as a $\L$-algebra in 
$F_1$, since all the operations of $(R(f);\Lal)$ are already present in the signature of the latter. It follows that $R(f)$ is e-interpretable as a $\L$-algebra in $(N,M;f,\La)$. 

Finally we prove Item 1, i.e.\ that $Z(\Sym(f))$, $R(f)$, $N$, and $M$ are all simultaneously either  finite, or all simultaneously infinite. We first claim that, in general, if $\Theta$ is an  associative commutative unitary ring, then any finitely generated faithful $\Theta$-module $K$  is  infinite if and only if $\Theta$ is  infinite. Indeed, if $K$ is finite, then ${\it End}_{\Theta}(K)$ is finite  as well, because ${\it End}_{\Theta}(K)$ embeds as a $\Theta$-module into $K^n$, for some $n$ (see Remark \ref{r: Ends_are_tuples}). Since $K$ is a faithful $\Theta$-module, there exists an embedding $\Theta \hookrightarrow {\it End}_{\Theta}(K)$, and hence $\Theta$ is finite as well. On the other hand, if $K$ is infinite, then, since $K$ is finitely generated, there must exist  $k\in K$ such that the set $\{\theta k \mid \theta \in \Theta\}$ is infinite, hence $\Theta$ is infinite.   The claim follows.

Observe that both $N$ and $M$ are faithful $R(f)$-modules, and that $N$ is also a  faithful $Z(\Sym(f))$-module. We  claim that both these modules are finitely generated.  Indeed,  let $\L_N=\{\l \in \L \mid \l n =0 \ \hbox{for all} \ n\in N\}$, and define $\L_M$ analogously. Then $N$ (resp.\ $M$) is a finitely generated faithful $\L/\L_N$-module (resp.\ $\L/\L_M$-module). Using that $f$ is full and non-degenerate, one can see that $N$ is also a faithful $\L/\L_M$-module under the action $(\l + \L_M)x= \l x$. With this action, $\L/\L_M$ becomes a ring of scalars of $f$, and so by maximality of $R(f)$ we have $\L/\L_M\leq R(f)\leq Z(\Sym(f))$ as subrings of $\End(N)$. Similar arguments yield $\L/\L_N\leq R(f)\leq Z(\Sym(f))$. Since  $N$ is finitely generated as a $\L/\L_N$-module, it is also finitely generated  as a $R(f)$-module. Similarly, $N$ is finitely generated as  a $Z(\Sym(f))$-module, and $M$ is finitely generated as a $R(f)$-module. Observe also that all these modules are faithful since $Z(\Sym(f))$ and $R(f)$ embed in $\End(N)$ by construction. This completes the proof of the claim. Item 1 of the statement of the theorem follows now from  this and the previous claim. \end{proof}

\subsection{Arbitrary bilinear maps}\label{s: arbitrary_bilinear_maps}

In this subsection we keep the   assumption that $\L$ is a Noetherian associative commutative ring with identity (possibly infinitely generated).   Our next goal is to generalize Theorem \ref{maxringdefinable} to arbitrary bilinear maps. Given a map $\Lambda$-bilinear map between finitely generated $\L$-modules  $f:A\times B \to C$, we let  the left and right \emph{annihilators} of $f$ be, respectively,
\begin{align}\label{e: ann}
{\it Ann}_l(f)=\left\{a\in A \mid  f(a,y)=0 \ \hbox{for all} \ y\in B \right\},\nonumber\\ 
{\it Ann}_r(f)=\{b\in B \mid  f(x,b)=0 \ \hbox{for all} \ x\in A\}.\nonumber
\end{align}
\begin{theorem}\label{ringdefinablegen}
Let $f:A \times B\to C$ be a $\Lambda$-bilinear map between finitely generated $\L$-modules. Then there exists an associative, commutative, unitary $\L$-algebra  $\Theta$ such    that  $(\Theta;\Lal)$  is finitely generated as a $\L$-module and it is e-interpretable  in $F=(A,B,C;f, \La)$. Moreover, 
\begin{enumerate}

\item In case that $\L$ is a field or the ring $\mbb Z$, then $(\Theta; \Lr)$ is e-interpretable in $(A, B, C; f, \Lg)$. 

\item $\Theta$ is   infinite if and only if both $\L$-modules $\langle f(A,B) \rangle_\L$ and $A/\Ann_l(f) \times B/\Ann_r(f)$ are  infinite, respectively. Here $\langle f(A,B) \rangle_\L$ denotes the $\L$-submodule of $C$ generated by the set $\{f(a,b) \mid a\in A, b\in B\}$. 
\end{enumerate}
\end{theorem}

The proof of this result relies on constructing from $f$  a suitable full non-degenerate bilinear map of the form $g: X\times X \to Y$, so that we can apply  Theorem \ref{maxringdefinable} to it. 
Observe that $f$ induces a full non-degenerate $\L$-bilinear map
\begin{equation}
f_1: A/\Ann_l(f) \times B/\Ann_r(f) \to \langle f(A,B) \rangle_\L.
\end{equation}
Let us denote $F=(A,B,C;f, \La)$, $A_1=A/Ann_l(f)$, $B_1=B/Ann_r(f)$, $C_1=\langle f(A,B) \rangle_\L$, and $F_1=(A_1, B_1, C_1; f_1, \La)$. Note that $A_1, B_1$ and $C_1$ are finitely generated since $A,B$ and $C$ are Noetherian modules. If $A_1=B_1$, then $f_1$ satisfies the hypothesis of   Theorem \ref{maxringdefinable}. Otherwise   consider the map
\begin{align}
\begin{split}\label{e: f2}
f_2: &\big(A_1 \times B_1\big) \times \big(A_1 \times B_1\big) \to C_1 \times C_1\\
&\big((a,b),(a',b')\big) \mapsto \big(f_1(a,b'),f_1(a',b)\big).
\end{split}
\end{align} 
One can easily check that $f_2$ is a full non-degenerate $\Lambda$-bilinear map between finitely generated $\L$-modules. Denote $F_2=(A_1\times B_1, C_1 \times C_1; f_2, \La)$.  
 Either $f_1$ or $f_2$ are of the desired form, hence Theorem \ref{maxringdefinable} can be applied to at least one of them. Moreover:
\begin{lemma}\label{l: F_1_is_e-interp}
Both $F_1$ and $F_2$ are e-interpretable in $F$. The same is true if one replaces $\La$ with $\Lg$ in $F_1, F_2$, and $F$.
\end{lemma}
\begin{proof}
 Let $S_A=\{a_1,\ldots,a_n\}$ and $S_B=\{b_1,\ldots,b_m\}$ be generating sets of $A$ and $B$, respectively. The submodules $\Ann_l(f)$ and $\Ann_r(f)$ 
%
%
%
are e-definable as sets 
in $F$ by the systems of equations $f(x,b_i)=0$, $i=1,\dots, m$, and $\wedge_i f(a_i, y)=0$, $i=1,\dots, n$, respectively. Here   $x$ and $y$ are variables taking values in $A$ and in $B$, respectively. 
An element $c\in C$ belongs to  $C_1=\langle f(A,B) \rangle_\L$ if and only if there exist elements $\lambda_{ij}$, $i=1,\dots, n$, $j=1,\dots, m$, such that
$$
c=\sum_{1\leq i,j\leq n,m} \l_{i,j}f(a_i,b_j)=\sum_{j=1}^m f\left(\sum_{i=1}^n \l_{i,j}a_i,b_j\right).
$$
 It follows that $C_1$ is e-definable as a set in $F$ by the equation $z=\sum_j f(x_j,b_j)$ on variables  $z$ and $X = \{x_j\mid j=1, \dots, n\}$. The variable $z$ takes values in $C$ and the variables from $X$ take values in $A$. 

The operations of $\Ann_l(f), \Ann_r(f)$ and $C_1$ are e-interpretable in $F$ because they are already present in the signature of $F$.  Hence by  Lemma \ref{factor_is_interpretable} and Remark \ref{r: factor_is_interpretable},  $(A_1;\La)$ and $(B_1;\La)$ are e-interpretable as $\L$-modules in $(A;\La)$ and $(B;\La)$, respectively. Moreover, from the proof of Lemma \ref{factor_is_interpretable} and the fact that the e-definitions in $F$ of $\Ann_l(f)$ and $\Ann_r(f)$  do not use multiplication by scalars, we have that $(A_1;\Lg)$ and $(B_1;\Lg)$ are e-interpretable in $(A;\Lg)$ and $(B;\Lg)$, respectively.

The preimage in $F$ of the graph of $f_1$ is e-definable in $F$ by the system consisting on the two equations $z = f(x,y)$ and $z = \sum_j f(x_j,b_j)$ on variables $z, x,y, X=\{x_1, \dots, x_m\}$ taking values in $C$, $A$, $B$, and $A$, respectively (note that the second equation is added to ensure that $z$ takes values in $C_1$). Again this equation does not use multiplication by scalars. We conclude that $F_1=\left( A_1, B_1, C_1; f_1, \La \right)$ is e-interpretable in $(A,B,C;f,\La)$, and that the same holds if one drops multiplication by scalars in both structures. 

Finally, we claim that $(A_1\times B_1; \La)$ is e-interpretable in  $(A_1,B_1; \La)$, and $(C_1\times C_1; \La)$ is e-interpretable in $(C_1;\La)$. Indeed, both $A_1\times B_1$ and $C_1 \times C_1$ are basic sets of $(A_1, B_1)$ and $C_1$, and so they are defined as sets by empty systems of equations. Similarly as before, the equations $z=f_1(x, y')$ and $z'=f_1(x', y)$ on variables $x,y,x',y', z, z'$ taking values in $A_1, B_1, A_1, B_1, C_1, C_1$, respectively, e-define the graph of $f_2$  in $F_1$. It follows that the whole two-sorted structure $F_2$ is e-interpretable in $F_1$, and also in $F$ by transitivity of e-interpretations.  Moreover,  in all e-interpretations we constructed, multiplication by scalars $\L$ in one structure is only used to e-interpret  multiplication by scalars $\L$ in the other structure. Hence $F_2$ is still e-interpretable in $F$  if  one  replaces $\La$ by $\Lg$ in the  $F_2$ and $F$.  \end{proof}

\begin{proof}[Proof of Theorem \ref{ringdefinablegen}] The result follows immediately  after using Theorem \ref{maxringdefinable} in order to e-interpret $(Z(\Sym(f_2));\Lal)$ or $(Z(\Sym(f_1));\Lal)$ in $F_2$ or $F_1$, depending on whether or not $A_1=B_1$, respectively. Items 1 and 2 are  a direct consequence of Items 1 and 3 of Theorem \ref{maxringdefinable}.
\end{proof}

\begin{remark}\label{r: choose_ring}
In Theorem \ref{ringdefinablegen} we  e-interpreted $Z(\Sym(f_2))$ in $F_2$ (or $Z(\Sym(f_1))$ in $F_1$ if $A_1=B_1$).  Alternatively one can also  e-interpret the largest ring of scalars $R(f_2)$  of $f_2$ in $F_2$ (similarly for $f_1$). This may have some advantages if one seeks to study the structure of $A,B,$ $C$ and $f$, because $R(f_2)$ is determined by  ``more properties'' of these  than $Z(\Sym(f_2))$. However, when it comes to the Diophantine problem, $Z(\Sym(f_2))$ is a more practical choice than $R(f_2)$, because it uses a simpler e-interpretation. For instance, as we have seen, if $\L$ is a field then one can drop multiplication by scalars in the e-interpretation of $(Z(\Sym(f_2));\Lr)$, whereas there is no apparent way to do the same with $(R(f_2);\Lr)$.
\end{remark}

\section{Rings and algebras over finitely generated associative commutative unitary rings}\label{s: section5}

The following lemma is a combination of the results obtained so far. It constitutes the main ``general tool'' presented in this paper.  We will explore its consequences throughout the rest of the section. In \cite{GMO_solvable} it is applied further to the area of group theory. We refer again to Subsection \ref{p: notation} for important notation and terminology conventions. 

\begin{lemma}\label{approach}
Let $\L$ be a associative commutative unitary ring,  let $f:A\times B\to C$ be a $\L$-bilinear map between finitely generated $\L$-modules, and write $C_1=\langle {\it Im}(f)\rangle_\L$. Suppose that $(A,B,C;f, \La)$ is e-interpretable in some structure $\mc{M}$. Then there exists an associative, commutative, unitary $\L$-algebra $R$ such that $R$ is finitely generated as a $\L$-module and $(R;\Lal)$ is e-interpretable in $\mc{M}$. Moreover, $R$ is infinite if and only if $C_1$ is  infinite. 

Furthermore, if $C_1$ is infinite and $\L$ is finitely generated, then there exists a ring of integers  $O$ of a number field or a  global function field such that $(O;\Lr)$ is e-interpretable in $(R;\Lr)$, and in $\mc{M}$. Additionally in this case:
\begin{enumerate}

\item If $\L$ has  positive characteristic $p$, then the ring of polynomials $(\mbb{F}_p[t];\Lr)$ is e-interpretable in  $\mc{M}$, and $\mc{D}(\mc{M})$ is undecidable.
\item If $\L=\mbb Z$ then $O$ is a ring of algebraic integers.
\end{enumerate}
If $\L$ is  $\mbb Z$ or a field, then the whole lemma  holds after replacing  $(A,B,C; f, \La)$ by $(A,B,C; f, \Lg)$ and $(R;\Lal)$ by $(R;\Lr)$, i.e.\ multiplication by scalars is not required.
\end{lemma}

\begin{proof}
By Theorem \ref{ringdefinablegen}, there exists an associative commutative unitary $\L$-algebra $R$ such that $R$ is finitely generated as a $\L$-module and $(R;\Lal)$ is e-interpretable in $(A,B,C;f, \La)$. Hence $(R;\Lal)$ is e-interpretable in $(R;\Lal)$  by transitivity of e-interpretations. The statement regarding the cardinality of $R$ follows from Item 2 of Theorem \ref{ringdefinablegen}.

Suppose that $C_1$ is infinite and that $\L$ is finitely generated. Then   $R$ is infinite and finitely generated as a ring. Hence,  by Theorem \ref{t: f_g_commutative_rings} there exists a  ring of  integers $O$ of  a number field or a  global function field such that $(O;\Lr)$  is e-interpretable in $(R;\Lr)$, and hence in $\mc{M}$.

If $\L$ has positive characteristic $p>0$, then so does  $R$, because it  is  a unitary algebra over $\L$. Hence $(\mbb{F}_p[t];\Lr)$ is e-interpretable in $\mc{M}$, by   Item 1 of Theorem \ref{t: f_g_commutative_rings} and by transitivity. If $\L=\mbb Z$, then $R$ is finitely generated as an abelian group, hence $O$ is a ring of algebraic integers by Item 3 of Theorem \ref{t: f_g_commutative_rings}.  

If $\L$ is $\mbb Z$ or a field, then $(R; \Lr)$ is e-interpretable in $(A,B,C;f,\Lg)$, by  Item 1 of Theorem \ref{ringdefinablegen}. Therefore if the latter is e-interpretable in $\mc{M}$, then the lemma holds after replacing $\Lal$ by $\Lr$ and $\La$ by $\Lg$. 
\end{proof}

\subsection{Rings and algebras which are finitely generated as modules}\label{s: module_finite_algebras}

Throughout this subsection, $\L$ denotes a finitely generated associative commutative unitary ring, possibly infinitely generated.

The following is one of the main results of the paper. The case $\L=\mbb Z$ will be considered separately afterwards. We recall that the language $\Lal$ of a $\L$-algebra  refers to the language of rings together with multiplication by any constant element from $\L$ (however variables cannot take values in $\L$), see Subsection \ref{p: notation} for more details.

\begin{theorem}\label{t: main_thm_algebras}
Let $R$ be a (possibly non-associative, non-commutative, and non-unitary)   algebra over a finitely generated associative commutative unitary ring $\L$. Suppose  that $R$ is finitely generated as a $\L$-module.  Then if,  $ R^2 = \langle \{xy\mid x\in R, y\in R\} \rangle_\L$ is infinite,  there exists a  ring of  integers $O$ of a number field or a  global function field such that $(O;\Lr)$  is e-interpretable in $(R; \Lal)$, and the Diophantine problem $\mc{D}(O;\Lr)$ is Karp-reducible to $\mc{D}(R;\Lal)$. Moreover: 
\begin{enumerate}
\item If $R^2$ is infinite and $\L$ has positive characteristic, then $(\mbb{F}_p[t];\Lr)$ is e-interpretable  in $(R;\Lal)$ for some prime integer $p$, and $\mc{D}(R;\Lal)$ is undecidable. 
\item If  $R^2 $ is finite and $\mc{D}(R; \La)$ is decidable, then $\mc{D}(R;\Lal)$ is decidable. 
\end{enumerate}
If $\L$ is a finite field, then all the above holds after replacing $(R;\Lal)$ by $(R;\Lr)$.  
\end{theorem}
\begin{proof}
The ring multiplication operation $\cdot$  of $R$ induces a $\L$-bilinear map between finitely generated $\L$-modules $\cdot: R \times R \to R$,  with  $\langle Im(\cdot) \rangle_{\L}=R^2$. Since the three-sorted structure $(R, R, R; \cdot, \La)$ is e-interpretable in $(R; \Lal)$, the  result, except Item 2, follows from Lemma \ref{approach}.

We now prove Item 2. Let $\Sigma$ be a system of equations in the $\L$-algebra $R$. By adding new variables and equations in an analogous way as done in the proof of Lemma \ref{equation_reduction}, we may rewrite that $\Sigma$ into an equivalent system of equations, still denoted $\Sigma$, of the following form: the new $\Sigma$ consists in a system of equations in the $\L$-module $(R;\La)$ (i.e.\ a system of $\L$-linear equations), in conjunction with  a system of equations of the form  $x_{1}=y_{1}z_1, x_{2}=y_{2} z_{2}, \dots, x_k = y_k z_k$ where the $x_1, \dots, x_k, y_1, \dots, y_k, z_1, \dots, z_k$ are variables taking values in $R$. Note that no variable appears more than once in this last part of the  system.


Note  that  $\Ann_l(\cdot)$ and $\Ann_r(\cdot)$ are finite index submodules of $R$,   by Item 2 of Theorem \ref{ringdefinablegen}.    Let $S_l=\{a_1, \dots, a_s\}$ and $S_r=\{b_1, \dots, b_t\}$  be full systems of coset representatives of $R/\Ann_l(\cdot)$ and $A/\Ann_r(\cdot)$, respectively. Let also $S_R$ be a finite generating set of $R$. 

For each variable $y\in \{y_1, \dots, y_k\}$, do the following:  choose a coset representative $a\in \{a_1, \dots, a_s\}$, and introduce a new variable $y'$. Then replace each occurrence of $y$ in $\Sigma$ by $a+ y'$. We now wish to make $y'$ to take values in $\Ann_l(\cdot)$.  The system of equations $\wedge_{r\in S_R} ur=0$ on the variable $u$ e-defines $\Ann_l(\cdot)$ in $R$. Hence, we add the system   $\wedge_{r\in S_L} y' r=0$ to $\Sigma$ in order to ensure that $y'$ takes values in $\Ann_l(\cdot)$. Notice that $\wedge_{r\in S_L} y' r=0$ is a system of equations in the $\L$-module $(R;\La)$.

We now proceed in an analogous  way with each variable $z\in \{z_1, \dots, z_k\}$, this time replacing $z$ by  $b+z'$ for some $b\in\{b_1, \dots, b_t\}$, and adding the system of equations $\wedge_{r\in S_R} rz'=0$ to $\Sigma$, in order to ensure that $z'$ takes values in $\Ann_r(\cdot)$.  

Let $\Sigma'$ be the resulting system of equations after making all the above transformations. Since there are finitely many coset representatives, the number of all  possible resulting systems $\Sigma'$ is finite. Let $\Sigma_1, \dots, \Sigma_m$ be all of them. It is clear that $\Sigma$ has a solution if and only if $\Sigma_i$ has a solution for some $i=1,\dots, m$.

We now prove that it is possible to decide algorithmically if each one of the $\Sigma_i$ has a solution or not, in which case our proof is concluded. Indeed, each $\Sigma_i$ consists in some equations in the $\L$-module $(R;\La)$, together with some equations of the form $x= (a+ y')(b+z')$, where $y'$ and $z'$ are variables  taking values in $\Ann_l(\cdot)$ and in $\Ann_r(\cdot)$, respectively. Hence each equation $x= (a+ y')(b+z')$ can be replaced by $x=ab$, which is an equation in $(R;\La)$. Thus $\Sigma_i$ is equivalent to a system of $\L$-linear equations. By assumption, we can algorithmically check if $\Sigma_i$ has a solution. 
\end{proof}

The following is  a particular case of the previous Theorem \ref{t: main_thm_algebras}. It is stated separately due to its independent interest.

\begin{theorem}\label{non_commutative_associative_version} 
Let $A$ be a ring  (possibly non-associative, non-commutative, and non-unitary). Assume that $A$  is finitely generated as an abelian group, and that $ A^2$ is infinite.  Then there exists a  ring of algebraic integers $O$ such that $(O;\Lr)$  is e-interpretable in $(A; \Lr)$, and $\mc{D}(O; \Lr)$ is Karp-reducible to $\mc{D}(A,\Lr)$. If otherwise $A^2$ is finite, 
then $\mc{D}(R;\Lr)$ is decidable. 
\end{theorem}

\begin{proof}The first part of the theorem follows in the same way as  Theorem \ref{t: main_thm_algebras}, taking $\L=\mbb Z$ and observing that here $O$ is a ring of algebraic integers by Item 2 of Lemma \ref{approach}. The last part  follows by Item 2 of  Theorem \ref{t: main_thm_algebras}, since here  $(A;\La)$ is  a finitely generated abelian group and $\mc{D}(A;\Lg)$ is decidable \cite{Ershov},  hence  $\mc{D}(A;\La)$ is decidable.
\end{proof}

\subsection{Finitely generated associative, commutative non-unitary rings and algebras}

 In this subsection we study associative commutative rings and algebras  that do not have an identity element. We begin with a  lemma that allows us to e-interpret a certain unitary algebra.

\begin{lemma}\label{l: non_unit_to_unit}
Let $A$ be a finitely generated associative commutative non-unitary algebra   over a (possibly infinitely generated) associative commutative unitary  ring $\Theta$. Then there exists   an associative commutative unitary ring $\L$ and an associative commutative unitary $\L$-algebra  $B$ such that $B$ is finitely generated as a $\L$-module, and $(B;\Lal)$ is e-interpretable in $(A;\Lal)$.
%
Additionally,   $A^2$ is infinite if and only if $B$ is infinite.  
\end{lemma}

\begin{proof}
Define $\L=\Theta \oplus A$ to be the set of formal sums of elements from $\Theta$ and from $A$, equipped with the natural addition and multiplication operations. More precisely,   $\L=\Theta \oplus A=\{(\theta,a)\mid \theta \in \Theta, a \in A\}$ and  $\L$ is endowed with the following ring operations: $(\theta_1,a_1)+(\theta_2,a_2)=(\theta_1+\theta_2,a_1+a_2)$ and $(\theta_1,a_1)\cdot(\theta_2,a_2)=(\theta_1\theta_2,\theta_1a_2+\theta_2a_1 + a_1a_2)$.  We have that $\L$ is an associative commutative unitary ring.

 Moreover, $\L$ acts naturally by endomorphisms on $A$, i.e.\ $(\theta, a) \cdot a' = \theta a' + a a'$ for all $(\theta, a)\in \L$ and all $a'\in A$. With this action  $A$ is a $\L$-algebra. During this proof we write $A_\L$ and $A_\Theta$ to refer to $A$ seen as a $\L$-algebra or as a $\Theta$-algebra, respectively. The operation of multiplication by a given scalar $\theta\oplus a\in \L$ is e-interpreted in $(A_\Theta;\Lal)$ by the  equation $y=\theta x+ a x$ on variables $y,x$ taking values in $A$. Hence $(A_\L;\Lal)$ is e-interpretable in  $(A_\Theta;\Lal)$. Suppose that $A_\Theta$ is generated as a $\Theta$-algebra by $n$ elements $S_A=\{a_1, \dots, a_n\}$. Then, since each element of $A$ can be written as a linear combination of monomials from $\Theta[a_1, \dots, a_n]$, we have that $A_\L$ is generated as a $\L$-module by $S_A$, since for all $x\in A$ there exists $y_1, \dots, y_n \in  \L$  such that $x=\sum_i y_i a_i$. Hence $A$ is  finitely generated as a $\L$-module.

The ring multiplication of $A_\L$ is a $\L$-bilinear map between finitely generated $\L$-modules $\cdot:A\times A\to A$. Moreover, $(A_\L,A_\L, A_\L; \cdot, \La)$ is e-interpretable in $(A_\L;\Lal)$, which in turn is e-interpretable in $(A_\Theta; \Lal)$. Hence, by the first part of Lemma \ref{approach}, and by transitivity of e-interpretations, there exists an associative commutative unitary $\L$-algebra  $D_\L$, which is finitely generated as a $\L$-module, and $(D_\L;\Lal)$ is e-interpretable in $(A_\Theta;\Lal)$. 

Finally, we note that  $(A_\L)^2 = \langle {\it Im}(\cdot) \rangle_\L$ is  infinite if and only if $D_\L$  is also  infinite, by Lemma \ref{approach}.
\end{proof}

We are ready to study associative commutaty non-unitary algebras.  We convene that the \emph{characteristic} of a non-unitary ring $A$ is defined as the maximum positive integer $n$ such that $nx=0$ for all $x\in R$. 


\begin{theorem}\label{t: non_unit_algebras}
 Let $L$ be a finitely generated associative commutative non-unitary algebra over a finitely generated associative commutative unitary ring $\Theta$, with $L^2$ infinite. Then there exists a  ring of  integers $O$ of a number field or of a  global function field such that $(O;\Lr)$  is e-interpretable in $(L; \Lal)$, and $\mc{D}(O;\Lr)$ is Karp-reducible to $\mc{D}(L;\Lal)$. 

Moreover, if $\Theta$ has positive characteristic, then $O$ is the ring of integers of a global function field,  $(\mbb{F}_p[t];\Lr)$ is e-interpretable  in $(L;\Lal)$ for some prime integer $p$, and $\mc{D}(L;\Lal)$ is undecidable.  
\end{theorem}
\begin{proof}
Let $B$ be the associative commutative unitary $\L$-algebra, finitely generated as a $\L$-module, given by  Item 1 of Lemma \ref{l: non_unit_to_unit}, where $\L=\Theta \oplus B$. Suppose that $L^2$ is infinite. By this same lemma, $B$ is infinite, and since $B$ is unitary we have that $B^2=B$ is infinite as well. Note further that if $\Theta$ has positive characteristic then so does  $\L$. The result now follows by transitivity of e-interpretations and by Theorem \ref{t: main_thm_algebras} applied to $B$. 
\end{proof}


Finally, we formulate and slightly modify the previous Theorem \ref{t: non_unit_algebras} for the particular case when $L$ is a ring, i.e.\ when $\Theta = \mbb{Z}$.

\begin{theorem}\label{t: non_unit_commutative_ring}
Let $A$  be a finitely generated associative commutative non-unitary ring, with   $A^2$   infinite. Then there exists a ring of integers $O$ of a number or a global function field such that $(O;\Lr)$ is e-interpretable in $(A; \Lr)$, and $\mc{D}(O;\Lr)$ is Karp-reducible to $\mc{D}(A;\Lr)$. 

Moreover, if $A$ has positive characteristic, then the following holds: $O$ is the ring of integers of a global function field; the ring of polynomials  $(\mbb{F}_p[t]; \Lr)$ is e-interpretable in $A$ for  prime integer $p$; and $\mc{D}(A; \Lr)$ is undecidable.
%
\end{theorem}

\begin{proof} 
Let $n$ be the characteristic of $A$. Then $A$ is a  finitely generated $\mbb{Z}/n\mbb{Z}$-algebra, where if $n=0$ we understand that $\mbb{Z}/n\mbb{Z}=\mbb{Z}$.   Moreover, multiplication by scalars from $\mbb{Z}/n\mbb{Z}$ is e-interpretable in $(A,\Lr)$ since $(k + n\mbb{Z})\cdot x = x + \overset{k}{\dots}+x$ for any $x\in A$ and any equivalence class $(k+n\mbb{Z})\in \mbb{Z}/n\mbb{Z}$ with representative $0\leq k\leq n-1$. Hence $(A, \Lal)$ is e-interpretable in $(A,\Lr)$.  The result now follows by applying the previous Theorem \ref{t: non_unit_algebras} on $(A, \Lal)$, and by transitivity of e-interpretations.  
\end{proof}

\subsection{Finitely generated rings and algebras satisfying an infiniteness condition}\label{s: f_g_algebras_with_ininiteness_condition}

 In this subsection $\L$ denotes an associative commutative unitary ring, possibly infinitely generated.

We next apply  Theorems \ref{t: main_thm_algebras} and \ref{non_commutative_associative_version} to  certain classes of non-commutative finitely generated  rings and algebras $L$ which are not necessarily finitely generated as modules. The approach consists in e-defining an ideal $I_n$  in $L$ that contains ``enough'' products of at least $n$ elements of $L$ (for example, the ideal generated by all such products), so that the quotient $L/I_n$  is infinite and finitely generated as a module. Then it suffices to apply the results from Section \ref{s: module_finite_algebras}, together with transitivity of e-interpretations. This approach presents two challenges: 
\begin{enumerate}
\item $I_n$ can be difficult to e-define if $L$ is non-associative  (hence Definition \ref{d: rnb}).

\item  $L/I_n$ may be finite. For instance, if $L$ is unitary (i.e.\ it has an identity element) one cannot simply take $I_n$ to be the ideal generated by all products of $n$ elements of $L$, since then $I_n=L$, and $L/I_n = 0$ (hence the next definition).
\end{enumerate}

\begin{definition}\label{d: I_n}
Let $L$ be a $\L$-algebra, and let $T$ be a generating set of $L$. If $L$ is non-unitary then we let $I_n(T)$ or $I_n$ denote the $\L$-ideal generated by all  products of $n$ elements of $T$. 

If $L$ is unitary then we let $I_n(T)$ or $I_n$ denote the $\L$-ideal generated by all  products of $n$ elements of $T\setminus\{\lambda\cdot 1 \mid \lambda\in \Lambda\}$, where $1$ denotes the multiplicative identity of $L$. 
%
%
\end{definition}

Throughout the rest of the section, $L$ denotes a finitely generated $\L$-algebra, possibly  non-associative, non-commutative,  non-unitary, and not finitely generated as a $\L$-module. We fix a finite set $T=\{a_1, \dots, a_m\}$ as in Definition \ref{d: I_n}, and we define the ideals $I_n$ accordingly.

\begin{definition}\label{d: rnb}
Let  $L$ be a $\L$-algebra and let $T=\{a_1, a_2, \dots\}$ be a generating set of $L$. Then $L$ is called  \emph{left-normed-generated} with respect to $T$ if, for all $n\geq 1$,  $I_n(T)$ is generated as a $\L$-module by all left-normed products $\{(a_{i_1}(a_{i_2}(\dots(a_{i_{k-1}} a_{i_{k}}) \dots ))) \mid k\geq n, \ 1\leq i_1, \dots, i_k\}$.   
\end{definition}

Notice that any associative algebra is left-normed-generate with respect to any generating set.

\begin{lemma}\label{l: define_Rn}
Suppose that $L$ is a  left-normed-generated $\L$-algebra with respect to some finite generating set, and let $n\geq 1$. Then  $(L/I_n; \Lr)$ and $(L/I_n; \Lal)$ are e-interpretable in $(L; \Lr)$ and $(L; \Lal)$, respectively. Moreover, $L/I_n$  is  a $\L$-algebra which is finitely generated as a $\L$-module. 
\end{lemma}
\begin{proof}
Let $T=\{a_1, \dots, a_m\}$ be a finite generating set of $L$. Then each element of $I_n$ is a finite sum of elements of the form  \begin{equation}\label{e: define_In}\l (a_{i_1}(a_{i_2}( \dots (a_{i_{k-1}} a_{i_{k}})\dots))), \quad \l\in \L,\quad k\geq n.\end{equation}  Hence each element as in \eqref{e: define_In}  can be written as $(a_{i_1}(\dots(a_{i_{n-1}}y)\dots))$ for some $y\in L$  in the non-unitary case. Consequently, if $L$ is non-unitary, then $I_n$ is e-definable as a set in $(L;\Lr)$ by the equation
\begin{equation}\label{e: non_unitary}
x= \sum_{1\leq i_1, \dots, i_{n-1}\leq m} (a_{i_1}(\dots(a_{i_{n-1}} y_{i_1, \dots, i_{n-1} })\dots ))
\end{equation}
on variables $x$ and $\{y_{i_1, \dots, i_{n-1}}\}$ (the e-definition is in $(L; \Lr)$ because it makes no use of multiplication by scalars $\L$). Observe however that \eqref{e: non_unitary} would not work in the case that $L$ is unitary, since a solution to \eqref{e: non_unitary}  may yield $x \in I_{n-1}$, for example if each $y_{i_1, \dots, i_{n-1}}$ is a $\Lambda$-multiple of $1$. This can be solved  by, in this case, taking the equation 
\begin{equation}\label{e: unitary}
x= \sum_{1\leq i_1, \dots, i_{n}\leq m} (a_{i_1}(\dots(a_{i_{n-1}} (a_{i_n} y_{i_1, \dots, i_{n} }))\dots ))
\end{equation}
instead of \eqref{e: non_unitary}. 
Hence, in both cases $(L/I_n; \Lr)$ and $(L/I_n; \Lal)$ are  e-interpretable in $(L;\Lr)$ and $(L; \Lal)$ by Lemma \ref{factor_is_interpretable}.


Finally, note that $L/I_n$ is generated as a $\L$-module by the projection of all products of less than $n$ elements of $T$, together with $1+I_n$ if $L$ is unitary. It follows that $L/I_n$ is finitely generated as a $\L$-module.
\end{proof}

We now state the main result of this subsection. The ideals $I_n$ are defined with respect to any set $T$ satisfying the conditions of Definition \ref{d: I_n}. 

\begin{theorem}\label{t: finitely_generated_algebras}
Let $L$ be a   finitely generated  algebra (possibly  non-associative, non-commutative and non-unitary) over a finitely generated associative commutative unitary ring $\L$. Suppose  that $L$ is  left-normed-generated with respect to some finite generating set, and that $(L/I_{n})^2$ is infinite for some $n\geq 1$. Then there exists a  ring of  integers $O$ of a number field or a  global function field such that $(O;\Lr)$ is e-interpretable in $(L; \Lal)$, and $\mc{D}(O; \Lr)$ is Karp-reducible to $\mc{D}(L; \Lal)$. Moreover:
\begin{enumerate}

\item If $\L$ has positive characteristic $p$, then $(\mbb{F}_p[t];\Lr)$ is e-interpretable in $(L; \Lal)$, and $\mc{D}(L;\Lal)$ is undecidable.

\item  If $L$ is a ring (i.e.\ $\L=\mbb{Z}$) then $O$ is a ring of algebraic integers.
    
\end{enumerate}
If $\L$ is $\mbb Z$ or a finite field then all the above holds after replacing $(L;\Lal)$ by $(L;\Lr)$.
\end{theorem}
%
\begin{proof}
By Lemma \ref{l: define_Rn}, $L/I_n$ is a  $\L$-algebra, it is finitely generated as a $\L$-module, and it is e-interpretable as an algebra in $(L;\Lal)$. The same result states that $(L/I_n, \Lr)$ is e-interpretable as a ring in  $(L;\Lr)$. By hypothesis, we can take $n$ so that $(L/I_n)^2$ is infinite. Now the result  follows by Theorems \ref{t: main_thm_algebras} and \ref{non_commutative_associative_version} applied to $L/I_n$, and by transitivity of e-interpretations. 
\end{proof}

We next apply the previous theorem to the class of Lie algebras, which are popular examples of non-associative, non-commutative and non-unitary algebras. First, we prove that Lie algebras are left-normed-generated.

\begin{lemma}\label{l: Lie_algebra_is_right_normed_gen}
 Any countably generated Lie algebra  $L$ is left-normed-generated with respect to any countable generating set. 
\end{lemma}
\begin{proof}
Let $A=\{a_1, a_2, \dots\}$ be a generating set of $L$. In \cite{Right_normed_basis_for_free_Lie_algebras} it is proved that any free Lie algebra $F=F(b_1, b_2, \dots)$ generated by $B=\{b_1, b_2, \dots\}$ is freely generated as a module by a subset of the set $B=\{(b_{i_1}(b_{i_2}(\dots(b_{i_{k-1}} b_{i_{k}}) \dots ))) \mid k\geq 1, \ 1\leq i_1, \dots, i_k \}$.  In particular, it is left-normed-generated with respect to $B$. Let $\pi: F\to L$ be the natural projection of $F$ onto  $L$ sending $b_i$ to $a_i$ for all $i\geq 1$. We denote by $I_n^F$ and $I_n^L$ the ideals of $F$ and $L$, respectively, generated by all products of at least $n$ elements from $B$ and from $A$, respectively. Observe that  $\pi(I_n^F) = I_n^L$. 

For each $n\geq 1$, let $S_n$ be a subset of  $\{[b_{i_1},[b_{i_2},[\dots, [b_{i_{k-1}}, b_{i_{k}}] \dots ]]] \mid k\geq n, \ 1\leq i_1, \dots, i_k\}$ such that the ideal $I_n^F$ of $F$ is generated by all $\L$-multiples of $S_n$. Then $\pi(I_n^F) = I_n^L$ is generated by all $\L$-multiples of $\pi(S_n)$. Now $\pi(S_n)$ is a subset of   $$\{[\pi(b_{i_1}),[\pi(b_{i_2}),[\dots, [\pi(b_{i_{k-1}}), \pi(b_{i_{k}})] \dots ]]] \mid k\geq n, \ 1\leq i_1, \dots, i_k \},$$ hence $L$ is left-normed generated with respect to $\pi(B)=A$.
\end{proof}

The next two corollaries follow immediately from Theorem \ref{t: finitely_generated_algebras} and  Lemma \ref{l: Lie_algebra_is_right_normed_gen}. 
By $[L/I_n,L/I_n]$ we denote the $\L$-submodule of $L/I_n$ generated by   $\{[x,y]\mid x,y\in L/I_n\}$. 

\begin{corollary}\label{t: Lie}
Let $L$ be a  finitely generated  Lie $\L$-algebra. Assume that $[L/I_n,L/I_n]$ is infinite for some $n\geq 1$, and that $\L$ is finitely generated.  Then the conclusions of Theorem \ref{t: finitely_generated_algebras}  hold for $L$.
\end{corollary}

\begin{corollary}\label{c: free}
Let $F$ be a finitely generated  free associative $\L$-algebra (possibly non-commutative and non-unitary) or a free Lie algebra of rank at least $2$, with $\L$ finitely generated. Then the conclusions of Theorem \ref{t: finitely_generated_algebras} hold for $F$.
\end{corollary}
\begin{proof}
If $F$ is a free Lie algebra, let $I_n$ be taken with respect to any free generating set of $F$, $n\geq 1$. If $F$ is a free associative algebra freely generated by $\{1, a_1, \dots, a_m\}$, let $I_n$ be taken with respect to $\{a_1, \dots, a_m\}$, $n\geq 1$.  In both cases $(F/I_n)^2$ is infinite for all $n\geq 3$.  Therefore the result follows from Theorem \ref{t: finitely_generated_algebras} and the previous Corollary \ref{t: Lie}. 
\end{proof}

Corollary \ref{c: free} complements Romankov's \cite{Romankov_eqns_2}, and Kharlampovich and Miasnikov's  \cite{ KM_free_algebras, KM_free_Lie_algebras} papers on free algebras. In the first reference \cite{Romankov_eqns_2}, it is proved that $\mc{D}(F;\Lr)$ is undecidable for many types of free rings $F$. In particular, it is proved that the algebras of Corollary \ref{c: free} have undecidable Diophantine problem if $\L=\mbb Z$.  In  \cite{KM_free_algebras, KM_free_Lie_algebras} it is proved that $\mc{D}(F;\Lr)$ is undecidable if $\L$ is an arbitrary field, and $F$ is a free associative non-commutative unitary algebra, or a free Lie algebra of rank at least $3$.  Note that an infinite field is necessarily infinitely generated, hence our Corollary \ref{c: free}  does not intersect with \cite{KM_free_algebras, KM_free_Lie_algebras}. 

\subsection{Undecidability of first-order theories}

The \emph{first-order theory} $T$ (or \emph{elementary theory}) of a structure $M$ in a language $\mc{L}$  is the set of all first-order sentences in $\mc{L}$ that are true in $M$.  One says that $T$ is decidable if there exists an algorithm that, given a sentence $\phi$ in $\mc{L}$, determines if $\phi$ is true in $M$ or not, i.e.\ if $\phi$ belongs to $T$. If such an algorithm does not exist then $T$ is said to be undecidable.  

The first-order theory \emph{with constants} of $M$ in the language $\mc{L}$ is the set of first order sentences that are true in $M$, allowing the use of any constant element from $M$ in the sentences.

Noskov proved in \cite{Noskov_rings} that the first-order theory of an infinite finitely generated associative commutative unitary ring is undecidable in the language of rings with constants. In particular this is true for the ring of integers of any number field or  global function field. Hence using transitivity of e-interpretations and Proposition \ref{Diophantine_reduction}  we immediately obtain the following:
\begin{theorem}\label{t: first_order}
Let $L$ be a ring or an algebra satisfying the hypotheses of one of the  Theorems \ref{t: f_g_commutative_rings}, \ref{t: main_thm_algebras}, \ref{non_commutative_associative_version}, \ref{t: non_unit_algebras}, \ref{t: non_unit_commutative_ring},  \ref{t: finitely_generated_algebras},  or Corollaries \ref{t: Lie} and \ref{c: free}. Assume $L^2$ is infinite. Let $\mc{L}$ denote $\Lr$ if $L$ is a ring (this includes the case when $L$ is a $\mbb{Z}$-algebra) or if $L$ is an algebra over a field and it does not satisfy the hypothesis of Theorem \ref{t: non_unit_algebras}. Otherwise  let $\mc{L}=\Lal$.  Then the first-order theory of $L$ in the language $\mc{L}$ with constants is undecidable. 
%
%
%
\end{theorem}

\section{Appendix: finitely generated associative commutative unitary rings}\label{appendix}

In this section we prove Theorem \ref{t: f_g_commutative_rings}, which we restate  next. The definition of rank we use is given in Definition \ref{l: trank}.

\begin{theorem}\label{t: f_g_commutative_rings_appendix}

Let $R$ be an infinite finitely generated commutative ring with identity.
Then there exists a ring of integers $O$ of a number or a global function field such that $(O;\Lr)$ is e-interpretable in $(R;\Lr)$,  and $\mc{D}(O;\Lr)$ is polynomial-time Karp-reducible to $\mc{D}(R;\Lr)$. Moreover, one of the following holds:

\begin{enumerate}
    \item If $R$ has positive characteristic $p> 0$, then the following holds: $O$ is the ring of integers of a global function field; the ring of polynomials  $(\mbb{F}_p[t];\Lr)$ is e-interpretable in $(R;\Lr)$ for some variable $t$; and $\mc{D}(R;\Lr)$ is undecidable.
    \item  If  $R$ has zero characteristic and it has infinite rank then the same conclusions as above hold:   $O$ is the ring of integers of a global function field; the ring of polynomials  $(\mbb{F}_p[t];\Lr)$ is e-interpretable in $(R;\Lr)$ for some prime $p$ and variable $t$; and $\mc{D}(R;\Lr)$ is undecidable. 
    \item  If  $R$ has zero characteristic and it has finite rank $n$, then $O$ is a ring of algebraic integers, and $\mc{D}(R;\Lr)$ is undecidable provided that $\mc{D}(O;\Lr)$ is undecidable. Additionally, $K$ is a field extension of $\mbb{Q}$ of degree at most $n$. 
\end{enumerate}
\end{theorem}

As we mentioned, this result is almost identical to Theorem 7.1 in Eisentraeger thesis \cite{phd_eisentrager} (see Theorem \ref{t: phd_eisentrager} in this paper for the statement in \cite{phd_eisentrager}). Since we are interested in having an e-interpretation of a ring of integers $O$ in $R$, rather than just a reduction of the Diophantine problems, and for completeness, we provide a proof in this appendix.

We postpone the proof for later in this section. Next we introduce some intermediate notions and results that we will need.  

\begin{proposition}\label{c: integral_closure_e_interpretable}
Let $R$ be a finitely generated integral domain whose field of fractions $K$ is a number or a global function field. In the case that $K$ is a global function field of characteristic $p$, assume that $R$ contains $\mbb{F}_p[t]$.   
Then the ring of integers $O$ of $K$ is e-interpretable in $R$. Furthermore, if $K$ has positive characteristic $p$, then  $\mbb{F}_p[t]$ is e-interpretable in $R$, and the Diophantine problem $\mc{D}(R)$ is undecidable.
\end{proposition}

  We proceed to provide a proof. First, we review some necessary notions and results.   Proposition \ref{c: integral_closure_e_interpretable}  is essentially a restatement of some of the results from \cite{Shla_book}. There, instead of e-interpretability, the notion of Dioph-generation is used: 

\begin{definition}[Definition 2.1.5 \cite{Shla_book}]\label{d: dioph_gen}
Let $R_1$ and $R_2$ be two integral domains with fields of fractions $F_1$ and $F_2$, respectively. Assume that neither $F_1$ nor $F_2$ is algebraically closed. Let $F$ be a finite extension of $F_2$ such that $F_1\subseteq F$. Further, assume that for some integers $k$ and $m$ there exists a base $\{\omega_1, \dots, \omega_k\}$ of $F$ over $F_2$ and a polynomial $f(a_1, \dots, a_k, b, x_1, \dots, x_m)$ with coefficients in $R_2$ such that  $f(a_1, \dots, a_k, b, x_1, \dots, x_m)=0$ implies that $b\neq 0$, and
\begin{align*}
R_1 = \{&\sum_{i=1}^k t_i \omega_i\mid \exists  a_1, \dots, a_k, b, x_1, \dots, x_m\in R_2,\\ &bt_1=a_1, \dots, bt_k=a_k,  f(a_1, \dots, a_k, b, x_1, \dots, x_m)=0   \}.
\end{align*}
Then we say that $R_1$ is \emph{Dioph-generated} over $R_2$.
\end{definition}

The corresponding statement is the following:

\begin{theorem}[\cite{Shla_book}]\label{c: integral_closure_e_interpretable_dioph_gen}
%
Let $R$ be a finitely generated integral domain whose field of fractions $K$ is a number or a global function field. In the case that $K$ is a global function field of characteristic $p$, assume that $R$ contains $\mbb{F}_p[t]$.   
Then the ring of integers  $O$ of $K$ is Dioph-generated over $R$. Furthermore, if $R$ has positive characteristic $p$, then  $\mbb{F}_p[t]$ is Dioph-generated over $R$, and the Diophantine problem $\mc{D}(R)$ is undecidable.
\end{theorem}

Next we use the above Theorem \ref{c: integral_closure_e_interpretable_dioph_gen} in order to prove Proposition \ref{c: integral_closure_e_interpretable_dioph_gen}. It suffices to prove a suitable equivalence between the notions of e-interpretability and of Dioph-generation.

The next definition will be used only in an auxiliary manner in the next Lemma \ref{l: dioph_gen_ans_e_interpretation}.

\begin{definition}[Definition 2.1.1 \cite{Shla_book}]\label{d: field_dioph}
Let $R$ be an integral domain with field of fractions $F$. Let $k,m$ be positive integers and let $A\subseteq F^k$ be some subset of $F^k$. Assume further that there exists a polynomial $f(a_1, \dots, a_k, b, x_1, \dots, x_m)$ with coefficients in $R$ such that, for all $a_1, \dots, a_k, b, x_1, \dots, x_m\in R$, we have $f(a_1, \dots, a_k, b, x_1, \dots, x_m)= 0 \Rightarrow b\neq 0$, and 
\begin{align*}
A=\{ &(t_1, \dots, t_k)\in F^k\mid \exists  a_1, \dots, a_k, b, x_1, \dots, x_m\in R,\\ &bt_1=a_1, \dots, bt_k = a_k, f(a_1, \dots, a_k, b, x_1, \dots, x_m)=0\}.
\end{align*}
Then $A$ is said to be \emph{field-Diophantine} over $R$.
\end{definition}

Next we provide a condition for  Dioph-generation to imply e-definability.

\begin{lemma}\label{l: dioph_gen_ans_e_interpretation}
 Let $R_1, R_2$ be integral domains with $R_1 \subseteq R_2$. Then $R_1$ is Dioph-generated over $R_2$ if and only if $R_1$ admits a $1$-dimensional e-definition in $R_2$. 
\end{lemma}
\begin{proof}
Suppose  $R_1$ is Dioph-generated over $R_2$. By Corollary 2.1.10 in \cite{Shla_book}, $R_1$ is field-Diophantine over $R_2$.  Now by Lemma 2.1.2 in \cite{Shla_book} (taking $A$ to be $R_1$, $R$ to be $R_2$, and $k=1$) we have that $R_1$ is e-definable
in $R_2$  (note that in \cite{Shla_book} an e-definition  is called \emph{Diophantine definition}, see 1.2.1 \cite{Shla_book}).  Moreover, from the proof  of Lemma 2.1.2 in \cite{Shla_book} we see immediately that the e-definition is $1$-dimensional. 

Now assume $R_1$ admits a 1-dimensional e-definition in $R_2$. Then again by Lemma 2.1.2 in \cite{Shla_book}  we have that $R_1$ is field-Diophantine over $R_2$. Moreover, the proof of this lemma shows that  $R_1$ is field-Diophantine over $R_2$ taking $k=1$ in the notation of Definition \ref{d: field_dioph} (and taking $A$ to be $R_1$, and  $R$ to be $R_2$).  We now claim that $R_1$ is Dioph-generated over $R_2$. Indeed, it suffices to take, following the notation in Definition \ref{d: dioph_gen}, $F=F_2$, the basis $\{1\}$ of $F$ over $F_2$, and the polynomial $f$ from the field-Diophantine definition of $R_1$ in $R_2$.
\end{proof}

\begin{proof}[Proof of Proposition \ref{c: integral_closure_e_interpretable}]
It follows immediately from Theorem \ref{c: integral_closure_e_interpretable_dioph_gen}, Lemma \ref{l: dioph_gen_ans_e_interpretation}, and Remark \ref{r: contained_e_def_implies_e_interp}.
\end{proof}

We will need the following observation:

\begin{remark}\label{r: rank_properties}
Let $R$ be a countable commutative ring of finite rank and positive characteristic $k$. Then $R$ is finite: indeed, this follows from one of Prüfer theorems, as in this case $R$ is a bounded abelian group since $kR=0$  (see Theorem 5.2 in \cite{Fuchs}).

Furthermore, if $1 \to R_1 \to R_2 \to R_3 \to 1$ is a short exact sequence of rings, then the rank of $R_2$ is at least the rank of $R_1$, and at most the rank of $R_1$ plus the rank of $R_3$ (see Exercise 3, Chapter 3.4 in \cite{Fuchs}). Finally,  if $A$ is a finitely generated $R$-module, and $R$ has finite rank, then $A$ as an abelian group also has finite rank (this follows from the fact that an abelian group $B$ has rank $k$ if and only if $k$ is the largest integer such that $B$ contains a subgroup $B_0$ which is the direct sum of $k$ cyclic groups, and for all $b \in B$ there exists an integer $n\neq 0$ such that $nb\in B_0$). 
\end{remark}

We are ready to prove  Theorem \ref{t: f_g_commutative_rings_appendix}.

\begin{proof}[Proof of Theorem \ref{t: f_g_commutative_rings_appendix}]

Throughout the proof we will  use the facts that e-interpretability is transitive, by Proposition \ref{interpretation_transitivity}; and that the quotient by any ideal of a Noetherian ring $R$ is e-interpretable in $R$,  by Lemma \ref{interpretations_in_rings}. More precisely, we  successively replace $R$ by  appropriate quotients of $R$ until obtaining an infinite finitely generated subring $R'$ of a number or a global function field $K$. We then use Proposition \ref{c: integral_closure_e_interpretable} from the previous section, and obtain first an e-interpretation of $O_K$  in $R'$ for some number or global function field $K$, and then an e-interpretation in $R$ by the aforementioned transitivity property and  Lemma \ref{interpretations_in_rings}. Moreover, since $R'$ is a quotient of $R$, Items 1 and 3 of the statement follow rather quickly. Item 2, the case when $R$ has infinite rank and zero characteristic, requires an extra intermediate  step where a suitable quotient of the form $R/pR$ is found, for some prime $p$.

\noindent\emph{Step 1: Reduction to integral domains.} Let $R$ be a finitely generated infinite commutative ring. Suppose first that $R$ is not an integral domain. We will find a quotient of $R$ which is an infinite finitely generated integral domain and which is e-interpretable in $R$. Let $N=\{x \in R\mid x^m = 0 \ \text{for some} \ m\in \mbb{N}\}$ be the nilradical of $R$, i.e.\ the ideal formed by all nilpotent elements of $R$. Equivalently, $N$ is the intersection of all minimal prime nonzero ideals of $R$. There are finitely many such ideals $q_1, \dots, q_n$ in a Noetherian ring  (see Theorem 87 of \cite{Kaplanski}), hence $N=q_1\cap \dots \cap q_n$. We claim that   $n\geq 1$. Indeed,  $R$ contains at least one nonzero maximal ideal, since otherwise $R$ would be a finitely generated ring which is a field, and so  $R$ would be finite (see Exercise 6 in Chapter 7 of \cite{Atiyah}). Since maximal ideals are prime, we have $n\geq 1$.    

We now claim that there exists $i$ such that $R/q_i$ is   infinite. We also claim that if $R$ has infinite rank, then there exists $i$ such that $R/q_i$ has infinite rank (in particular,  $R/q_i$ is infinite by Remark \ref{r: rank_properties}).  Indeed, note first that $R/N$ admits an embedding into the direct sum  $R/q_1 \oplus \dots \oplus R/q_n$ via the well-defined map $r + N \mapsto (r
+q_1) \oplus \dots \oplus (r+q_n)$. Hence, if all $R/q_i$ are finite, then $R/N$ is finite. If all $R/q_i$ have finite rank, then also $R/N$ has finite rank by Remark \ref{r: rank_properties}.

Now, the $R/N$-module $N^i/N^{i+1}$ is finitely generated as a $R/N$-module (since it is finitely generated as a ring) for all $i=1,\dots, n-1$. It follows that if $R/N$ is finite then $N^i/N^{i+1}$ is also finite. The same is true for the rank: if $R/N$ has finite rank, then $N^i/N^{i+1}$ also has finite rank by Remark \ref{r: rank_properties}.  Since   $N^m=0$ for some $m$, it follows that if $R/N$ is finite then also $R$ is  finite, a contradiction. Similarly, it follows that if $R/N$ has finite rank then so does $R$, by Remark \ref{r: rank_properties}.  Hence the claim is proved.  

Note that if $R$ has zero characteristic and finite rank then $R/q_i$ has also zero characteristic, since otherwise it would be finite.

In views of the previous arguments, we can assume from now on that $R$ is an infinite finitely generated integral domain.  Note that, by Remark \ref{r: trank_domains}, now the notion of rank is the same as vector space dimension (for positive characteristic) and of $\mbb{Z}$-module dimension (for zero characteristic). 

\medskip

\noindent\emph{Step 2: The case of infinite rank and zero characteristic.} 
We now assume  that $R$ is a finitely generated  integral domain of zero characteristic and infinite rank. We will e-interpret in $R$ an infinite finitely generated integral domain of positive characteristic. This will allow us to assume, in the next steps of the proof, that $R$ either has finite rank and zero characteristic, or infinite rank and positive characteristic. In particular, this reduces the hypothesis of Item 2 in the statement of the theorem to the hypothesis of Item 1.

Note that for every prime integer $p$,  $R/pR$ is e-interpretable in $R$ by Lemma \ref{interpretations_in_rings}. 
Hence this step will be complete once we prove that there exists a prime $p$ such that $R/pR$ is  infinite.

  We first claim that $R$ must contain  a transcendental element over $\mbb{Q}$ (we identify $R$ with its embedding in its field of fractions, which is a field extension of $\mbb{Q}$). Indeed, assume not, so that every element of $R$ is a root of a polynomial with integer coefficients. In particular, each element in a finite generating set of $R$, say $r_1, \dots, r_\ell$, is the root of some  polynomial in $\mbb{Z}[x]$. Since $R$ is  generated as a ring by $r_1,\ldots,r_\ell$, the field of fractions of $R$ is generated as a field by  $r_1,\ldots,r_\ell$. Since all $r_i$ are algebraic over $\mbb{Q}$ ($i=1, \dots, \ell$), the field of fractions of $R$ is a finite field extension of $\mbb{Q}$, i.e.\ it is a finite-dimensional $\mbb{Q}$-vector space. It follows that  $R$ has finitely many $\mbb{Z}$-linearly independent elements, and so $R$ has finite rank, a contradiction. The claim is proved.
  
  Now pick a transcendental element $x\in R$. Then the subring of $R$  generated by $1$ and $x$ is isomorphic to  $\mbb{Z}[x]$. We identify this subring with $\mbb{Z}[x]$. Given  an integer prime $q$ let $\phi_q:R\to R/qR$ be the natural quotient map. We will show that $\phi_q(\mbb{Z}[x])$ is infinite for some $q$. This will imply immediately that $R/qR$ is infinite as well, and hence this step will be complete. 

Let $\phi_q|_{\mbb{Z}[x]}$ be the restriction of $\phi_q$ on $\mbb{Z}[x]$. We will find a prime $q$ such that $\ker(\phi_q|_{\mbb{Z}[x]}) = q\mbb{Z}[x]$, from where it follows that $\phi_q(\mbb{Z}[x]) \cong  \mbb{F}_q[x]$ is infinite. 

Indeed, first note that $\ker(\phi_q|_{\mbb{Z}[x]}) = qR \cap \mbb{Z}[x]$.  Define $A=\{r\in R\mid nr\in \mbb{Z}[x]\text{ for some }n\in\mbb{N}\setminus\{0\}\}$. We have $\ker(\phi_q|_{\mbb{Z}[x]}) = qR \cap \mbb{Z}[x] = qA\cap\mbb{Z}[x]$. If $A=\mbb{Z}[x]$ it follows that $\ker(\phi_q|_{\mbb{Z}[x]})=q\mbb{Z}[x]$ as required. Hence assume $\mbb{Z}[x]\subsetneq A$.  Note that $A$ is a finitely generated subring of $R$ and that it can be identified with a subring of $\mbb{Q}[x]$. Hence any element of $\mbb{Q}[x]$ can be written as $p(x)/n$ for some $p(x)\in\mbb{Z}[x]$ and some $n\in \mbb{N}\setminus\{0\}$. Let $a_1, \dots, a_k$ be a finite generating set of $A$ as a ring, and let $p_1(x), \dots, p_k(x)$ and $n_1, \dots, n_k$ be polynomials from $\mbb{Z}[x]$ and non-zero integers, respectively, such that   $a_i =p_i(x)/n_i$ for all $i=1, \dots, k$. Since $\mbb{Z}[x]\subsetneq A$, at least one of the $n_i$ is larger than $1$. Let $q$ be a prime integer that is coprime with all $n_i$. It follows that every element of $A$ can be written in the form $p(x)/n$ where $n\geq 1$ is coprime with $q$, and $p(x)\in \mbb{Z}[x]$. Then any element from $qA$ which belongs to $\mbb{Z}[x]$ must belong to $q\mbb{Z}[x]$. Hence  $\ker(\phi_q|_{\mbb{Z}[x]}) = qA\cap \mbb{Z}[x]=q\mbb{Z}[x]$, as required.

\medskip

\noindent\emph{Step 3: Reduction to Krull dimension 1.} From now on we assume that either the hypothesis of Item 1 or of Item 3 of the statement of the theorem hold. Hence $R$ is an infinite fintely generated integral domain either of finite rank and zero characteristic, or of infinite rank and positive characteristic. The Krull dimension  of $R$ is the largest integer $k$ for which there exists a proper ascending chain of prime ideals $p_0<p_1<\ldots<p_k<R$. Such $k$ is finite under our assumptions  (see Section 8.2.1 of \cite{Eisenbud}).  It is not possible that $k= 0$, since in this case $R$ would be a finitely generated Artinian domain (see Proposition 9.1 in \cite{Eisenbud}), and thus a finitely generated field (see Proposition 8.30 of \cite{Clark}), a contradiction because, as referred to earlier, a finitely generated ring which is a field is necessarily finite.  Hence $k\geq 1$.  We may assume that $k=1$, since if $k\geq 2$ then   $R/p_{k-1}$ is a finitely generated integral domain, e-interpretable in $R$, and of Krull dimension $1$. The latter implies that $R/p_{k-1}$ is infinite. This implies that $R/p_{k-1}$  has finite rank and zero characteristic, or infinite rank and positive characteristic, depending on which of these two properties $R$ satisfies, respectively. 

\medskip

\noindent \emph{Step 4: Reduction to a subring of a number or a global function field.}
Assume $R$ is a finitely generated infinite integral domain of Krull dimension $1$, either of infinite rank and positive characteristic, or of finite rank and zero characteristic.
We  claim that   one of the following hold:
\begin{enumerate}
\item $R$ is a subring of  a number field (if $R$ has zero characteristic). This is proved in 2.2 of \cite{Noskov_rings}.
\item There exists a prime integer $p$ and a transcendental element $t\in R$ over $\mbb{F}_p$ such that $\mbb{F}_p[t]\subseteq R$  and $R$ is integral over  $\mbb{F}_p[t]$. It follows that $R$ is a subring of a finite field extension $K$ of $\mbb{F}_p(t)$, with $\mbb{F}_p[t] \subseteq R$. In particular, $R$ has positive characteristic.

This follows  from the Noether normalization lemma (Theorem A1 of Chapter 8.2 in \cite{Eisenbud}), which states  that any finitely generated $k$-algebra is a finitely generated module over $k[y_1, \dots, y_d]$, where $k$ is any field and $d$ is the Krull dimension of the algebra. Hence in our case $R$ is a finitely generated $\mbb{F}_p[t]$-module, and so it is integral over $\mbb{F}_p[t]$.

\end{enumerate}

\medskip

\noindent\emph{Step 5: Reduction to rings of integers.}  Assume $R$ satisfies Item 1 or Item 2 of the previous step. Then the field of fractions $K$ of $R$ is a number or  a global function field. Since $R$ is finitely generated,   Proposition \ref{c: integral_closure_e_interpretable} implies that the ring of integers $O_K$ of $K$ is e-interpretable in $R$ (note that Item 2 above grants us the requirement that $R$ contains $\mbb{F}_p[t]$). By transitivity (Proposition \ref{Diophantine_reduction}), $O_K$ is e-interpretable in $R$, and therefore $\mc{D}(O_K)$ is  Karp-reducible to $\mc{D}(R)$ (Proposition \ref{Diophantine_reduction}).

If $R$ has finite rank $n$, then it has zero characteristic, because it is infinite.  Hence, $R$ is a subring of a number field, and $O_K$ is a ring of algebraic integers. Moreover,  since $R$ as a $\mbb{Z}$-module has dimension $n$, we have that $K$ is an $n$-dimensional $\mbb{Q}$-vector space, i.e.\ $K$ is field extension of $\mbb{Q}$ of degree $n$. 

If  $R$ has characteristic $p>0$, then  Proposition \ref{c: integral_closure_e_interpretable}, and transitivity of e-interpretations and reduction of Diophantine problems (Propositions \ref{interpretation_transitivity} and \ref{Diophantine_reduction}) yield that $\mbb{F}_p[t]$ is e-interpretable in $R$, and that $\mc{D}(R)$ is undecidable.

\medskip

\emph{Step 6: Conclusion.}  Let $R$ be the ring given initially in the statement of the theorem, and let $O_K$ be the ring of integers obtained in the previous Step 5. As discussed at the beginning of the proof,  $O_K$ is e-interpretable in $R$ and $\mc{D}(O_K)$ is Karp-reducible to $R$. Moreover, we have, following each one of the previous Steps 1 through 5, that each one of Items 1, 2, and 3 in the statement hold: indeed, Item 2 reduces to Item 1,  and in the rest of cases the  fact that $R$ has zero or positive characteristic does not change throughout all steps. Hence  the   Items 1 and 3 hold due to Step 5.
\end{proof}

\bibliographystyle{plain} 
\bibliography{bib.bib}

\end{document}